\documentclass[12pt]{amsart}
\usepackage[margin=1.0in]{geometry}
\usepackage{amsmath,amssymb,amsfonts,graphicx,color,fancyhdr,psfrag,comment,enumerate}
\usepackage[latin1]{inputenc}
\usepackage[hyperpageref]{backref}
\usepackage[colorlinks=true, pdfstartview=FitV, linkcolor=blue, citecolor=blue, urlcolor=blue]{hyperref}
\allowdisplaybreaks

\usepackage{epstopdf}
\usepackage{mathrsfs}
\usepackage{epsfig}
\usepackage{hyperref}
\usepackage{ifthen}
\usepackage{float}
\usepackage{enumerate}
\usepackage{mathtools}
\usepackage[bottom]{footmisc}
\usepackage{tikz}
\usetikzlibrary{matrix,shapes,arrows,positioning,chains}
\usepackage{amssymb,amsmath,amsfonts}
\usepackage{bbm}

\numberwithin{equation}{section}

\newtheorem{theorem}{Theorem}[section]

\newtheorem{lemma}[theorem]{Lemma}
\newtheorem{definition}[theorem]{Definition}
\newtheorem{example}[theorem]{Example}
\newtheorem{remark}[theorem]{Remark}
\newtheorem{corollary}[theorem]{Corollary}

\newcommand{\CZ}{Calder\'on-Zygmund}
\newcommand{\X}{\mathbb{R}^{n_1+n_2}}

\DeclarePairedDelimiter\floor{\lfloor}{\rfloor}
\DeclareMathOperator*{\esssup}{ess\,sup}

\def\essinf{\mathop{\rm ess \, inf}}

\begin{document}
\title[Grushin pseudo-multipliers]{Sparse bounds for pseudo-multipliers associated to Grushin operators, I}

\author[S. Bagchi, R. Basak, R. Garg and A. Ghosh]
{Sayan Bagchi \and Riju Basak \and Rahul Garg \and Abhishek Ghosh}

\address[S. Bagchi]{Department of Mathematics and Statistics, Indian Institute of Science Education and Research Kolkata, Mohanpur--741246, West Bengal, India.}
\email{sayan.bagchi@iiserkol.ac.in}

\address[R. Basak]{Department of Mathematics, Indian Institute of Science Education and Research Bhopal, Bhopal--462066, Madhya Pradesh, India.}
\email{rijubasak52@gmail.com}

\address[R. Garg]{Department of Mathematics, Indian Institute of Science Education and Research Bhopal, Bhopal--462066, Madhya Pradesh, India.}
\email{rahulgarg@iiserb.ac.in}

\address[A. Ghosh]{Department of Mathematics, Indian Institute of Science Education and Research Bhopal, Bhopal--462066, Madhya Pradesh, India.} 
\curraddr{Tata Institute of Fundamental Research, Centre For Applicable Mathematics, Bangalore--560065, Karnataka, India.}
\email{abhi170791@gmail.com}



\begin{abstract}
In this article, we prove sharp quantitative weighted $L^p$-estimates for Grushin pseudo-multipliers satisfying H\"ormander's condition as an application of pointwise domination of Grushin pseudo-multipliers by appropriate sparse operators.
\end{abstract}
\maketitle

\tableofcontents

\section{Introduction} \label{sec:intro} 
The theory of pseudo-differential operators has always been of immense interest in Harmonic analysis. This present article is the first of a series of articles where we investigate quantitative weighted norm inequalities for pseudo--multipliers corresponding to Grushin operators. In order to describe our results, let us first review the relevant literature concerning pseudo-differential operators on the Euclidean space as well as Hermite pseudo-multipliers.

Given $m \in L^{\infty} \left( \mathbb{R}^{n} \times \mathbb{R}^{n} \right) $, consider the pseudo-differential operator $m(x, D)$ defined by 
$$ m(x, D)f(x) := (2\pi)^{-n/2} \int_{\mathbb{R}^{n}} m(x, \xi) \widehat{f}(\xi) e^{i x \cdot \xi} \, d\xi, $$
for Schwartz class functions $f$ on $\mathbb{R}^{n}$, where $\widehat{f}$ denotes the Fourier transform of $f$ which is defined by $\widehat{f}(\xi) = (2\pi)^{-n/2} \int_{\mathbb{R}^{n}} f(x) e^{- i x \cdot \xi} \, dx.$ We call $m$ to be the symbol of the pseudo-differential operator $m(x, D)$. When the function $m$ does not depend on the space variable $x$, the associated operator $m(D)$ is indeed a Fourier multiplier operator. 

Given $\sigma\in\mathbb{R}$, $0\leq \rho \leq 1, \, 0\leq \delta< 1$, let $\mathscr{S}^{\sigma}_{\rho,  \delta}(\mathbb{R}^n)$ represent the class of smooth functions $m:\mathbb{R}^n\times \mathbb{R}^n\to \mathbb{C}$ such that for all $\alpha, \beta \in \mathbb{N}^n$, 
\begin{align}
|\partial^{\alpha}_{x}\partial_{\xi}^{\beta}m(x, \xi)|\leq C_{\alpha, \beta}(1+|\xi|)^{\sigma-\rho |\beta|+\delta |\alpha|}.   
\label{symbol:euclidean}
\end{align}
It is well known that pseudo-differential operators with symbols belonging to $\mathscr{S}^{\sigma}_{1, 0}(\mathbb{R}^n), \sigma\leq 0,$ fall into the realm of generalized Calder\'on-Zygmund operators, that is, their kernels satisfy  H\"ormander's condition. Sharp $L^p$-boundedness results for pseudo-differential operators are attributed to the seminal work of Fefferman \cite{Fefferman-Israel-Journal}. For details on the classical development in the theory of pseudo-differential operators, we refer to \cite{Coifman-Meyer-pseudo-1978, Treves-pseudodifferential-Vol1Book,TaylorPseudodifferentialBook81}. In this article, we are concerned with an analogue of symbol classes $\mathscr{S}^{0}_{1, \delta}$, for $0\leq \delta< 1$, in the context of pseudo-multipliers associated to Grushin operators. One of our primary objectives is to obtain minimal smoothness assumptions on the symbol function such that the associated Grushin pseudo--multiplier operators fall into the category of Calder\'on--Zygmund operators on spaces of homogeneous type in the sense of Coifman--Weiss \cite{Coifman-Weiss-book-1971}. Henceforth, for the sake of brevity, we shall write in short ``a space of homogeneous type" to denote such a space. 

Recall that an important aspect of the study of multiplier theorems is to obtain their counterparts in the context of weights. To begin with, we recall the work of  Kurtz--Wheeden \cite{Kurtz-Wheeden-weighted-multipliers-TAMS}. They studied weighted $L^p$-estimates for Mihlin--H\"ormander type Fourier multipliers, deducing also the best class of Muckenhoupt $A_p$ weights depending on the order of differentiability of the multiplier function. For pseudo-differential operators, there is already an extensive literature on weighted inequalities. We mention a few here. Miller \cite{Miller-pseudodifferential-TAMS1982} initiated the study of the weighted boundedness of pseudo-differential operators. He employed the Fefferman--Stein sharp maximal function $\mathcal{M}^{\sharp}$ and obtained pointwise estimates of the following form: Let $m\in \mathscr{S}^{0}_{1, 0}(\mathbb{R}^n)$, then for $1<r<\infty$, 
\begin{align}
\mathcal{M}^{\sharp}(m(x, D)f)(x) \lesssim_r M_{r}f(x),
\label{Miller-TAMS}
\end{align}
where $M_{r}f:=(M(|f|^r))^{1/r}$ with $M$ denoting the Hardy-Littlewood maximal function. Invoking good-$\lambda$-inequalities and reverse H\"older's inequality for $A_p$ weights, estimate \eqref{Miller-TAMS} implies that $m(x, D): L^p(\mathbb{R}^n, w)\to L^p(\mathbb{R}^n, w)$ is bounded for every $w\in A_{p}, 1<p<\infty$.

The approach explained above is a key step in producing weighted estimates for pseudo-differential operators and was also employed in the works of Fefferman\cite{Fefferman-Israel-Journal}, Chanillo--Torchinsky \cite{Chanillo-Torchinsky-weighted-pseudodifferential}, Michalowski--Rule--Staubach \cite{Michalowski-Rule-Staubach-Canad2012}, and many others. However, in order to obtain sharper quantitative dependence of $\|m(x, D)\|_{L^p(w)\to L^p(w)}$ on the $A_p$ characteristic ($[w]_{A_p}$) of the weight $w,$ one employs the modern technique known as ``sparse domination", which evolved during the development of the celebrated $A_2$-conjecture. $A_2$-conjecture was solved in complete generality by Hyt\"onen \cite{Hytonen-A2-2012}. We also mention the fundamental work of Lerner \cite{Lerner-A2-2013} where a comparatively simple proof of $A_2$-conjecture was obtained by proving norm domination of \CZ~operators by sparse operators. Subsequently, Lacey \cite{Lacey-A2-proof-Israel2017}, Conde Alonso--Rey \cite{Conde-Rey-MathAnn-2016}, Hyt\"onen \textit{et al} \cite{Hytonen-Roncal-Tapiola-Rough-homogeneous-Israel-2017} and Lerner \cite{Lerner-NYJM-2016} obtained pointwise domination of \CZ~operators by sparse operators. Thanks to their wide applications in obtaining sharp quantitative estimates, establishing sparse operator bounds for classical operators in Harmonic analysis is a growing area of research and the same has been applied in various contexts. For example, see \cite{Bernicot-Frey-Petermichl-Sparse-Beyond-APDE-2016, Conde-Culiuc-DiPlinio-Ou-Rough-Singular-APDE-2017,Lacey-Mena-Reguera-Sparse-Bochner-Riesz-JFAA-2019, Beltran-Cladek-sparse-pseudodifferential}, and references therein. 

Before we move further, let us also mention here that beyond the Euclidean space, analogues of pseudo-differential operators (or pseudo-multipliers) are also being studied in various other contexts. In a non-Euclidean setting, probably this is the first work where we are able to obtain pointwise estimates for pseudo-multipliers in terms of sparse operators. As far as (unweighted) $L^p$-boundedness is concerned, there has been a lot of recent interest in studying pseudo-multipliers in the set-up beyond the Euclidean spaces. For compact lie groups, we refer to \cite{RuzhanskyTurunenPseudoCompactGroupBook} and the references therein. For Heisenberg groups and more general graded Lie groups, we refer to \cite{BahouriFermanianGallagherPseudodifferentialHeisenberg, FischerRuzhanskyQuantizationBook, CardonaDelgadoRuzhanskyJGA2021} and the references therein. For pseudo-multipliers on certain homogeneous type spaces associated with a class of self-adjoint operators, we refer to \cite{BernicotFreyPseudodifferentialSemigroupOperators, GeorgiadisNielsenPseudodifferentialSelfAdjointOperators}. 

Originally introduced in \cite{EppersonHermitePseudo}, Hermite pseudo-multipliers were further studied by the first author and S. Thangavelu in \cite{BagchiThangaveluHermitePseudo}. Under some Mihlin--H\"ormander-type conditions on symbol functions, authors of \cite{BagchiThangaveluHermitePseudo} established the following weighted $L^p$-boundedness result for Hermite pseudo-multipliers. Let us denote by $\Delta$ the forward-difference operator defined by $\Delta m(x, k) = m(x, k+1) - m(x, k)$, and define $\Delta^j m = \Delta(\Delta^{j-1}m)$ for $j \geq 2$. Also, let us denote by $A_p(\mathbb{R}^n)$ the space of Muckenhoupt $A_p$ weights. 

\begin{theorem}[\cite{BagchiThangaveluHermitePseudo}] \label{thm:Bagchi-Thangavelu-1}
Let $m \in L^{\infty} \left( \mathbb{R}^{n} \times \mathbb{N}^{n} \right) $ be such that the Hermite pseudo-multiplier $m(x, H) \in \mathcal{B} \left( L^2 (\mathbb{R}^n) \right)$. Assume in addition that for some $N \in \mathbb{N}$, 
\begin{align*} 
\sup_{x \in \mathbb{R}^n} \left| \Delta^j m(x, k) \right| & \lesssim_{N} (2k + n)^{-j}, \quad \text{for all } \, j \leq N + 1, \\ 
\text{and} \quad \sup_{x \in \mathbb{R}^n} \left| \partial_{x_i} \Delta^j m(x, k) \right| & \lesssim_{N} (2k + n)^{-j}, \quad \text{for all } \, j \leq N \, \text{ and } \, 1 \leq i \leq n, 
\end{align*}
then the following results hold true. 
\begin{enumerate}
\item If $N \geq \floor*{n/2}$, then $m(x, H)$ extends to a bounded operator on $L^p(\mathbb{R}^{n}, w)$ for every $p \in (2, \infty)$ and $w \in A_{p/2} (\mathbb{R}^n)$. 

\item If $N \geq n$, then $m(x, H)$ extends to a bounded operator on $L^p(\mathbb{R}^{n}, w)$ for every $p \in (1, \infty)$ and $w \in A_p(\mathbb{R}^n)$. 
\end{enumerate}
\end{theorem}

In the next subsection, we recall some preliminaries on Grushin operators before presenting our results in detail. 

%%%%%%%%%%%%%%%%%%%%%%%%%%%%%%%%%
%%%%%%%%%%%%%%%%%%%%%%%%%%%%%%%%%
%%%%%%%%%%%%%%%%%%%%%%%%%%%%%%%%%
%%%%%%%%%%%%%%%%%%%%%%%%%%%%%%%%%
%%%%%%%%%%%%%%%%%%%%%%%%%%%%%%%%%
%%%%%%%%%%%%%%%%%%%%%%%%%%%%%%%%%
%%%%%%%%%%%%%%%%%%%%%%%%%%%%%%%%%
%%%%%%%%%%%%%%%%%%%%%%%%%%%%%%%%%
%%%%%%%%%%%%%%%%%%%%%%%%%%%%%%%%%

\subsection{Grushin pseudo-multipliers} \label{subsec:intro-grushin}
We write points in $\mathbb{R}^{n_1 + n_2}$ as $x = (x^{\prime}, x^{\prime \prime}) \in \mathbb{R}^{n_1} \times \mathbb{R}^{n_2}$. For each $\varkappa \in \{ 1, 2, 3, \ldots\}$, let us consider Grushin operators 
\begin{align} \label{def:general-grushin-operator-even-power}
G_{\varkappa} = - \Delta_{x^{\prime}} - V_{\varkappa} \left(x^{\prime} \right) \Delta_{x^{\prime \prime}}
\end{align}
with $V_{\varkappa} \left(x^{\prime} \right)$ denoting either $|x^{\prime}|^{2 \varkappa}$ or $\sum_{j = 1}^{n_1} {x_j^{\prime}}^{2 \varkappa}$. 

Grushin operators were introduced in \cite{Grushin70}. The operator $G_\varkappa$ is degenerate elliptic along the $n_2$-dimensional plane $\{0\}\times \mathbb{R}^{n_2}$. It is studied in various contexts related to Dirichlet problems in weighted Sobolev spaces, free boundary problems in partial differential equations etc. In particular, when $\varkappa=1$, Grushin operators are closely connected to the sub-Laplacian on the Heisenberg group (see, for example, \cite{Dziubanski-Jotsaroop-Hardy-BMO-Grushin}). Since in this article we are interested in studying pseudo-multipliers associated to Grushin operators, we start with their spectral decomposition.

For a Schwartz class function $f$ on $\mathbb{R}^{n_1 + n_2}$, let $f^\lambda$ denote (upto a dimensional constant multiple) its inverse Fourier transform in $x^{\prime \prime}$-variable, given by $f^\lambda(x^{\prime}) = \int_{\mathbb{R}^{n_2}} f (x^{\prime}, x^{\prime \prime}) e^{i \lambda \cdot x^{\prime \prime}} \, d x^{\prime \prime}$. It follows that 
$$ G_{\varkappa} f (x) = (2 \pi)^{- n_2} \int_{\mathbb{R}^{n_2}} e^{-i \lambda \cdot x^{\prime \prime}} H_{\varkappa} (\lambda) f^{\lambda} (x^{\prime}) \, d\lambda, $$
where, for $\lambda \neq 0$, $H_{\varkappa}(\lambda) = - \Delta_{x^{\prime}} + |\lambda|^2 |x^{\prime}|^{2 \varkappa}$ or $H_{\varkappa}(\lambda) = - \Delta_{x^{\prime}} + |\lambda|^2 \sum_{j = 1}^{n_1} {x_j^{\prime}}^{2 \varkappa} $, depending on the choice of $V_{\varkappa} (x^{\prime})$. 

In particular, when $\varkappa = 1$, operators $H_{1}(\lambda) = H(\lambda) = - \Delta_{x^{\prime}} + |\lambda|^2 |x^{\prime}|^2$, for $\lambda \neq 0$, are called the scaled Hermite operators on $\mathbb{R}^{n_1}$. In that case, we drop the suffix and denote the Grushin operator as $G$ itself. 

It is well known (see, for example, Theorems XIII.16, XIII.64, and XIII.67 in \cite{ReedSimonBookVolumeIV}) that there exists a complete orthonormal basis $ \{ h_{\varkappa, k} : k \in \mathbb{N} \}$ of $L^2 \left( \mathbb{R}^{n_1} \right)$ such that $H_{\varkappa} (1) h_{\varkappa, k} = \nu_{\varkappa, k} h_{\varkappa, k}$ with $0 < \nu_{\varkappa, 1} \leq \nu_{\varkappa, 2} \leq \nu_{\varkappa, 3} \leq \ldots$ and $\lim_{k \to \infty} \nu_{\varkappa, k} = \infty.$ It is straightforward to verify that for each $k \in \mathbb{N}$ and $\lambda \neq 0$, if we consider 
$$ h^{\lambda}_{\varkappa, k} (x^\prime) = |\lambda|^{\frac{n_1}{2 (\varkappa + 1)}} h_{\varkappa, k} \left( |\lambda|^{\frac{1}{\varkappa + 1}} x^\prime \right), $$
then $ \{ h^{\lambda}_{\varkappa, k} : k \in \mathbb{N} \}$ forms an orthonormal basis for $L^2(\mathbb{R}^{n_1})$, and that 
\begin{align*} 
G_{\varkappa} \left( h^{\lambda}_{\varkappa, k} (x^\prime) \, e^{i \lambda \cdot x^{\prime \prime}} \right) & = |\lambda|^{\frac{2}{\varkappa + 1}} \nu_{\varkappa, k} \,  h^{\lambda}_{\varkappa, k} (x^\prime) \, e^{i \lambda \cdot x^{\prime \prime}}. 
\end{align*} 
Thus, we can write the the spectral decomposition of the Grushin operator $G_{\varkappa}$ as follows: 
$$G_{\varkappa} f (x) = (2 \pi)^{- n_2} \int_{\mathbb{R}^{n_2}} e^{-i \lambda \cdot x^{\prime \prime}} \sum_{k \in \mathbb{N}} |\lambda|^{\frac{2}{\varkappa + 1}} \nu_{\varkappa, k} \left( f^{\lambda}, h^{\lambda}_{\varkappa, k} \right) h^{\lambda}_{\varkappa, k} (x^\prime) \, d\lambda.$$ 

As also done in \cite{Bagchi-Garg-1} (see Subsections $1.3$ and $2.1$ of \cite{Bagchi-Garg-1}), given a symbol function $m \in L^\infty \left( \mathbb{R}^{n_1 + n_2} \times \mathbb{R}_+ \right) $, it is natural to consider the Grushin pseudo-multiplier $m(x, G_{\varkappa})$, defined (densely) on $L^2 \left( \mathbb{R}^{n_1 + n_2} \right)$ by 
\begin{equation} \label{Gru-pseudo}
m(x, G_{\varkappa}) f(x) := (2 \pi)^{- n_2} \int_{\mathbb{R}^{n_2}} e^{-i \lambda \cdot x^{\prime \prime}} \sum_{k \in \mathbb{N}} m \left( x, |\lambda|^{\frac{2}{\varkappa + 1}} \nu_{\varkappa, k} \right) \left( f^{\lambda}, h^{\lambda}_{\varkappa, k} \right) h^{\lambda}_{\varkappa, k} (x^\prime) \, d\lambda. 
\end{equation}
When the function $m$ does not depend on the space variable $x$, the associated operator is said to be a Grushin multiplier and is denoted by $m(G_{\varkappa}).$ In the last decade, boundedness of spectral multipliers for more general Grushin-type operators has been extensively studied by various authors. See, for example, \cite{ DuongOuhabazSikoraWeightedPlancherel2002JFA, JotsaroopSanjayThangaveluRieszTransformsGrushin, MartiniSikoraGrushinMRL, MartiniMullerGrushinRevistaMath, Chen-Sikora-multipliers-Grushin-type-2013, Martini-Dallara-robust-approach-Grushin-multiplers-2020, Martini-Dallara-optimal-multipliers-Grushin-plane-II, Martini-Dallara-optimal-multipliers-Grushin-plane-I}, and references therein. 

Now, consider the following first order gradient vector fields: 
\begin{align} \label{first-order-grad}
X_{j} = \frac{\partial}{\partial x_j^{\prime}} \quad \textup{and} \quad X_{\alpha, k} = {x^{\prime}}^{\alpha} \frac{\partial}{\partial x_k^{\prime \prime}},
\end{align}
for $1 \leq j \leq n_1$, $1 \leq k \leq n_2$, and $\alpha \in \mathbb{N}^{n_1}$ with $|\alpha| = \varkappa$. 

Clearly, with either choice of potential $V_{\varkappa} (x^{\prime})$, the corresponding Grushin operator $G_{\varkappa}$ can be expressed as a negative sum of $X^2_{j}$'s and $X^2_{\alpha, k}$'s (see, for example, Section 3 of \cite{RobinsonSikoraMathZ2016}). Let us denote by $X$ the first order gradient vector field 
\begin{align} \label{first-order-grad-vector}
X := (X_j, X_{\alpha,k})_{1 \leq j \leq n_1, \, 1 \leq k \leq n_2, \, |\alpha| = \varkappa}. 
\end{align}

We write $n_0 = n_1 + n_2 \binom{\varkappa + n_1 - 1}{n_1 - 1}$, and consider symbol functions $m \in L^\infty \left( \mathbb{R}^{n_1 + n_2} \times \mathbb{R}_+ \right)$ which satisfy the following estimate for some $\rho, \delta \geq 0$: 
\begin{align} \label{eq:grushin-symb} 
\left| X^\Gamma \partial_{\eta}^l m(x, \eta) \right| \lesssim_{\Gamma, l} (1 + \eta)^{- (1 + \rho) \frac{l}{2} + \delta \frac{|\Gamma|}{2}} 
\end{align} 
for some $\Gamma \in \mathbb{N}^{n_0}$ and $l \in \mathbb{N}$. 

\begin{remark}
Condition \eqref{eq:grushin-symb} is motivated by Euclidean symbol classes $\mathscr{S}^{0}_{\rho,  \delta}(\mathbb{R}^n)$. Observe that the Grushin operator $G_{\varkappa}$ is homogeneous of degree two, that is, $G_{\varkappa}\circ D_{r}=r^2 (D_{r}\circ G_{\varkappa})$, where $D_{r}(x^{\prime}, x^{\prime\prime})=(r x^{\prime}, r^{1+\varkappa} x^{\prime\prime})$ are the non-isotropic dilations associated to $G_{\varkappa}.$ Thus, a natural analogue of $\mathscr{S}^{0}_{\rho, \delta}(\mathbb{R}^n)$ condition for Grushin pseudo-multiplier $\widetilde{m}(x, \sqrt{G})$ is 
$$ \left| X^\Gamma \partial_{\eta}^l \widetilde{m}(x, \eta) \right| \lesssim_{\Gamma, l} (1 + \eta)^{- \rho l + \delta |\Gamma|}. $$ 
It is technically convenient to work with the symbol $m(x, \eta) := \widetilde{m}(x, \sqrt{\eta}),$ which results into condition \eqref{eq:grushin-symb}. For more details on this aspect, we refer to \cite{Bagchi-Garg-1} (see the discussion immediately after the statement of Theorem $1.5$ as well as Remark $2.2$ in \cite{Bagchi-Garg-1}).
\end{remark}

Before moving on, let us mention that an analogue of the Calder\'on--Vaillancourt theorem for Grushin pseudo-multipliers for $\varkappa =1$ was studied in \cite{Bagchi-Garg-1}, establishing $L^2$-boundedness of Grushin pseudo-multipliers $m(x,G)$ for symbol functions $m$ satisfying condition \eqref{eq:grushin-symb} with $0 \leq \delta < \rho \leq 1$. In its full generality, the endpoint case of $0 \leq \delta = \rho < 1$ is still open. 

In this paper, we study pseudo-multiplier operators $m(x, G_{\varkappa})$ corresponding to symbols $m(x, \eta)$ satisfying condition of the type \eqref{eq:grushin-symb} with $\rho = 1$ and $0 \leq \delta < 1$. In this direction, following are our main results. Throughout this article, $Q:=n_{1}+(1+\varkappa) n_{2}$ denotes the homogeneous dimension associated to the operator $G_{\varkappa},$ see \eqref{ineq:doubling-condition-grushin-metric} for details.

\begin{theorem} \label{thm:pseudo-grushin-a=0-unweighted}
Let $m \in L^\infty \left( \mathbb{R}^{n_1+n_2} \times \mathbb{R}_+ \right)$ be such that 
\begin{align*} 
\left| \partial_{\eta}^l m(x, \eta) \right| \lesssim_{l} (1 + \eta)^{-l}
\end{align*}
for all $l \leq Q + 1$. Assume also that $m(x, G_{\varkappa})$ is bounded on $L^2 ( \mathbb{R}^{n_1 + n_2})$. Then $m(x, G_{\varkappa})$ is of weak type $(1, 1)$ and as a consequence bounded on $L^p(\mathbb{R}^{n_1+n_2})$ for every $1<p<2$. 
\end{theorem}

The above theorem implies $L^p$-boundedness in the range $1<p<2$ only, and we can not use duality to conclude boundedness in the range $2<p<\infty$ because the adjoint operator $m(x, G_{\varkappa})^*$ is not necessarily a Grushin pseudo-multiplier.

In the following two theorems, we establish weighted $L^p$-boundedness results for $m(x, G_{\varkappa})$ in the range $2<p<\infty$ and $1 < p < \infty$, depending on the assumed number of derivatives of the symbol function $m(x, \eta)$. Let us remark that part of our motivation comes from the work of \cite{DuongSikoraYanJFA2011}, where the authors studied weighted $L^p$-boundedness results for spectral multipliers corresponding to a non-negative self-adjoint operator on a space of homogeneous type (see Theorem 3.1 in \cite{DuongSikoraYanJFA2011}). Compared to the results of \cite{DuongSikoraYanJFA2011}, in our work we are only concerned with Grushin operators, but more importantly our methods are well suited for Grushin pseudo-multipliers as well.

We are now ready to state our main results on sparse domination for operators $m(x, G_{\varkappa})$. Let $\mathcal{S}$ denote a family of generalised-dyadic cubes in the context of the homogeneous type space, which we will explain in detail in Section \ref{subsec:sparse-domination}. We say a collection of measurable sets $S \subset \mathcal{S}$ to be an $\eta$-sparse family (for some $0<\eta<1$) if for every member $\mathcal{Q}\in {S}$ there exists a set $E_{\mathcal{Q}} \subseteq \mathcal{Q}$ such that $|E_{\mathcal{Q}}|\geq \eta |\mathcal{Q}|$, and the sets $\{E_{\mathcal{Q}}\}_{\mathcal{Q}\in {S}}$ are pairwise disjoint. Corresponding to a sparse family ${S}$ and $1\leq r<\infty$, we define the sparse operator as follows: 
\begin{align} \label{def:Sparse-operator} 
\mathcal{A}_{r, S}f(x) = \sum_{\mathcal{Q}\in S} \left( \frac{1}{|\mathcal{Q}|} \int_{\mathcal{Q}}|f|^r \right)^{1/r} \chi_{\mathcal{Q}}(x).
\end{align} 
We simply write $\mathcal{A}_{S}$ for $\mathcal{A}_{1, S}$.

\begin{theorem} \label{thm:pseudo-grushin-a=0-less-derivative}
For a fixed $0 \leq \delta < 1$, let $m \in L^\infty \left( \mathbb{R}^{n_1+n_2} \times \mathbb{R}_+ \right)$ be such that 
\begin{align*} 
\left| \partial_{\eta}^{l} m(x, \eta) \right| & \lesssim_{l} (1 + \eta)^{-l}, \quad \text{for all} \quad l \leq \floor*{Q/2} + 1, \\ 
\text{and} \quad \left| X_x \partial_{\eta}^{l} m(x, \eta) \right| & \lesssim_{l, \delta} (1 + \eta)^{-l + \frac{\delta}{2}}, \quad \text{for all} \quad l \leq \floor*{Q/2}.
\end{align*} 
Assume also that the operator $T = m(x, G_{\varkappa})$ is bounded on $L^2 ( \mathbb{R}^{n_1 + n_2})$. Then, for every compactly supported bounded measurable function $f$ there exists a sparse family $S \subset \mathcal{S}$ such that
\begin{align*}
|T f(x)| \lesssim_{T} \mathcal{A}_{2, S}f(x),
\end{align*}
for almost every $x \in \mathbb{R}^{n_1+n_2}$.
\end{theorem} 

\begin{theorem} \label{thm:pseudo-grushin-a=0-full-derivative}
For a fixed $0 \leq \delta < 1$, let $m \in L^\infty \left( \mathbb{R}^{n_1+n_2} \times \mathbb{R}_+ \right)$ be such that 
\begin{align*} 
\left| \partial_{\eta}^{l} m(x, \eta) \right| & \lesssim_{l} (1 + \eta)^{- l}, \quad \text{for all} \quad l \leq Q + 1, \\ 
\text{and} \quad \left| X_x \partial_{\eta}^{l} m(x, \eta) \right| & \lesssim_{l, \delta} (1 + \eta)^{-l + \frac{\delta}{2}}, \quad \text{for all} \quad l \leq Q.
\end{align*}
Assume also that the operator $T = m(x, G_{\varkappa})$ is bounded on $L^2 ( \mathbb{R}^{n_1 + n_2})$. Then, for every compactly supported bounded measurable function $f$ there exists a sparse family $S \subset \mathcal{S}$ such that
\begin{align*}
|T f(x)| \lesssim_{T} \mathcal{A}_{S}f(x),
\end{align*}
for almost every $x \in \mathbb{R}^{n_1+n_2}$.
\end{theorem}

%%%%%%%%%%%%%%%%%%%%%%%%%%%%%%%%%
%%%%%%%%%%%%%%%%%%%%%%%%%%%%%%%%%
%%%%%%%%%%%%%%%%%%%%%%%%%%%%%%%%%
%%%%%%%%%%%%%%%%%%%%%%%%%%%%%%%%%
%%%%%%%%%%%%%%%%%%%%%%%%%%%%%%%%%
%%%%%%%%%%%%%%%%%%%%%%%%%%%%%%%%%
%%%%%%%%%%%%%%%%%%%%%%%%%%%%%%%%%
%%%%%%%%%%%%%%%%%%%%%%%%%%%%%%%%%
%%%%%%%%%%%%%%%%%%%%%%%%%%%%%%%%%

\subsection{Methodology of the proof and organisation of the paper} \label{subsec:intro-methodology-proofs} 
Obtaining sharp quantitative estimates through sparse domination method has produced many fascinating results in recent times in the theory of singular integrals. In this article, one of our main objectives is to establish pointwise sparse domination for Grushin pseudo-multipliers. The approach of obtaining suitable end-point boundedness of some variants of the sharp maximal function and stopping time arguments have played a crucial role in obtaining sparse operator bounds in homogeneous type as well as in the non-homogeneous type spaces. In the present work, we also adapt this strategy. More precisely, in our work we are going to make use of the sparse domination principle developed in \cite{Lerner-Ombrosi-pointwaise-sparse2020, Lorist-pointwaise-sparse2021} (see Theorem \ref{thm:Lorist-general-sparse-principle}). There is a rich and ever growing literature on the existence of quantitative weighted estimates for sparse operators, which explains why one looks for sparse domination of various operators at hand. For our purposes, we state one such result in Theorem \ref{thm:Quantitative-bounds-sparse-operators}. So, in view of Theorem \ref{thm:Lorist-general-sparse-principle}, it boils down to establishing appropriate weak type boundedness of the operator $T$ that we are concerned with as well as an appropriate weak type boundedness of an associated grand maximal truncated operator $\mathcal{M}^{\#}_{T, s}$ (defined in \eqref{def:grand-maximal-truncated-operator}). 

Rest of the article is structured as follows.

\begin{itemize}
\item We collect the preliminary results relevant to our work in Section \ref{sec:prelim}. 

\medskip \item With assumptions on weighted kernel estimates with respect to the control distance associated to Grushin operators $G_{\varkappa}$, we state and prove our pointwise sparse domination results (Theorems \ref{thm:main-sparse}, \ref{thm:main-sparse-more-derivative} and \ref{thm:domination-maximal-M2}) for a general class of operators in Section \ref{sec:kernel-estimates-and-pointwise-sparse-domination}. The main idea is to show that the sharp grand maximal truncation operator $\mathcal{M}^{\#}_{T, s}$ (for a suitable choice of $s$) could be dominated by a maximal operator, and the exact type of the maximal operator is dictated by the assumptions on the weighted kernels, and this allows us to conclude appropriate end-point boundedness for $\mathcal{M}^{\#}_{T, s}$. In our proof, we also suitably make use of a mean-value inequality associated with the sub-Riemannian structure of the Grushin metric. It is an elementary
property satisfied by most reasonable sub-Riemannian metrics, but for the convenience of readers, we explain it in Section \ref{sec:prelim} in our context (see the discussion leading to the formulation of Lemma \ref{lem:general-grushin-Mean-value}). 

\medskip \item In Section \ref{sec:direct-pseudo-grushin}, we specialise to the case of the Grushin pseudo-multipliers. In order to do so, we prove the weighted (gradient) Plancherel estimates for the Grushin pseudo-multiplier operators, and with that our results (Theorems \ref{thm:pseudo-grushin-a=0-unweighted}, \ref{thm:pseudo-grushin-a=0-less-derivative} and \ref{thm:pseudo-grushin-a=0-full-derivative}) follow from results of Section \ref{sec:kernel-estimates-and-pointwise-sparse-domination}. Our proofs of weighted (gradient) Plancherel estimates are of self-interest, and we systematically document these in Subsection \ref{subsec:gradient-weighted-Plancherel-estimates}. While the ones for the integral kernels follow quite easily from results of the existing literature (\cite{AnhBuiDuongSpectralMultipliersBesovTriebelLizorkin,DuongOuhabazSikoraWeightedPlancherel2002JFA}) on kernels corresponding to the spectral multipliers, we are able to show that by making an appropriate use of the gradient estimates for the heat kernels (see estimates \eqref{est:derivative-bounds-heat-kernel-grushin}), it is possible to repeat ideas of \cite{AnhBuiDuongSpectralMultipliersBesovTriebelLizorkin,DuongOuhabazSikoraWeightedPlancherel2002JFA} in establishing analogous weighted Plancherel estimates for the gradients of the integral kernels of spectral multipliers. From this, the estimates for integral kernels of the pseudo-multipliers easily follow. 
\end{itemize}

%%%%%%%%%%%%%%%%%%%%%%%%%%%%%%%%%
%%%%%%%%%%%%%%%%%%%%%%%%%%%%%%%%%
%%%%%%%%%%%%%%%%%%%%%%%%%%%%%%%%%
%%%%%%%%%%%%%%%%%%%%%%%%%%%%%%%%%
%%%%%%%%%%%%%%%%%%%%%%%%%%%%%%%%%
%%%%%%%%%%%%%%%%%%%%%%%%%%%%%%%%%
%%%%%%%%%%%%%%%%%%%%%%%%%%%%%%%%%
%%%%%%%%%%%%%%%%%%%%%%%%%%%%%%%%%
%%%%%%%%%%%%%%%%%%%%%%%%%%%%%%%%%

\subsection{Notations and parameters} \label{subsec:notations} 
In the following list, we have written most of the notations and parameters that are being used in this paper. 
\begin{itemize} 
\item $\mathbb{N} = \{0, 1, 2, 3, \ldots\}$, and $\mathbb{R}_+ = [0, \infty)$. 

\item For $A, B >0$, by the expression $A \lesssim B$ we mean $A \leq C B$ for some $C > 0$. Whenever the implicit constant $C$ may depend on $\epsilon$, we write $A \lesssim_{\epsilon} B$. When $A \lesssim B$ and $B \lesssim A$, we write $A \sim B$.

\item For a Banach space $Y$, we write $\mathcal{B} \left( Y \right) $  for the Banach space of all bounded linear operators on $Y$ and $\| T \|_{op}$ denotes the operator norm of $T \in \mathcal{B} \left( Y \right)$. 

\item We write $x = (x^{\prime}, x^{\prime \prime}) \in \mathbb{R}^{n_1} \times \mathbb{R}^{n_2} \, = \mathbb{R}^{n_1 + n_2}$. 

\item For each $\varkappa \in \{ 1, 2, 3, \ldots\}$, the homogeneous dimension of the space $\mathbb{R}^{n_1 + n_2}$ with respect to the Grushin operator $G_{\varkappa}$ is $Q = n_1 + (1+\varkappa) n_2$. 

\item Whenever $\mu = (\mu_1, \ldots,\mu_{n_1})$ belongs to $\mathbb{N}^{n_1}$ we  write $|\mu| = \sum_{j=1}^{n_1}\mu_j$. For a general vector $c = (c_1, \ldots, c_{n_1})$ from $\mathbb{R}^{n_1}$, we have $|c| = (\sum_{j=1}^{n_1} |c_j|^2)^{1/2}$ and $|c|_1 = \sum_{j=1}^{n_1} |c_j|$. 

\item  $ \mathfrak{r}, R_0, \eta$ etc always correspond to elements of $\mathbb{R}_+$. 

\item  $ l, j, N$ etc always correspond to elements of $\mathbb{N}$. 

\item $\mu, \nu, \gamma, \alpha$ etc always correspond to elements of  $\mathbb{N}^{n_1}$. 

\item $\beta, \beta_j$ etc always correspond to elements of $\mathbb{N}^{n_2}$. 

\item $\Gamma, \Gamma_j $ etc always correspond to elements of $\mathbb{N}^{n_0}$, where $n_0 = n_1 + n_2 \binom{\varkappa + n_1 - 1}{n_1 - 1}$. In particular, when $\varkappa = 1$, we have $n_0 = n_1 + n_1 n_2$. 

\item  $\lambda $ represents the element of $\mathbb{R}^{n_2}$. 

\item When we talk about the case of the Euclidean space or for the Hermite operator, we always take $\mathbb{R}^n$ as our underlying space and use the indices as tabulated above with the convention of $n = n_1$. 
\end{itemize}  

%%%%%%%%%%%%%%%%%%%%%%%%%%%%%%%%%%%%%%%%%%%%%
%%%%%%%%%%%%%%%%%%%%%%%%%%%%%%%%%%%%%%%%%%%%%
%%%%%%%%%%%%%%%%%%%%%%%%%%%%%%%%%%%%%%%%%%%%%
%%%%%%%%%%%%%%%%%%%%%%%%%%%%%%%%%%%%%%%%%%%%%
%%%%%%%%%%%%%%%%%%%%%%%%%%%%%%%%%%%%%%%%%%%%%
%%%%%%%%%%%%%%%%%%%%%%%%%%%%%%%%%%%%%%%%%%%%%
%%%%%%%%%%%%%%%%%%%%%%%%%%%%%%%%%%%%%%%%%%%%%
%%%%%%%%%%%%%%%%%%%%%%%%%%%%%%%%%%%%%%%%%%%%%
%%%%%%%%%%%%%%%%%%%%%%%%%%%%%%%%%%%%%%%%%%%%%
%%%%%%%%%%%%%%%%%%%%%%%%%%%%%%%%%%%%%%%%%%%%%
%%%%%%%%%%%%%%%%%%%%%%%%%%%%%%%%%%%%%%%%%%%%%
%%%%%%%%%%%%%%%%%%%%%%%%%%%%%%%%%%%%%%%%%%%%%
%%%%%%%%%%%%%%%%%%%%%%%%%%%%%%%%%%%%%%%%%%%%%
%%%%%%%%%%%%%%%%%%%%%%%%%%%%%%%%%%%%%%%%%%%%%

\section{Preliminaries and basic results} \label{sec:prelim}

In this section, we write down the preliminary details relevant to the subject matter of this article.  

%%%%%%%%%%%%%%%%%%%%%%%%%%%%%%%%%
%%%%%%%%%%%%%%%%%%%%%%%%%%%%%%%%%
%%%%%%%%%%%%%%%%%%%%%%%%%%%%%%%%%
%%%%%%%%%%%%%%%%%%%%%%%%%%%%%%%%%
%%%%%%%%%%%%%%%%%%%%%%%%%%%%%%%%%
%%%%%%%%%%%%%%%%%%%%%%%%%%%%%%%%%
%%%%%%%%%%%%%%%%%%%%%%%%%%%%%%%%%
%%%%%%%%%%%%%%%%%%%%%%%%%%%%%%%%%
%%%%%%%%%%%%%%%%%%%%%%%%%%%%%%%%%

\subsection{Control distance for Grushin operators \texorpdfstring{$G_{\varkappa}$}{}} 
The control distance $\tilde{d} (x,y)$ (see equation (11) of \cite{RobinsonSikoraDegenerateEllipticOperatorsGrushinTypeMathZ2008}) associated with the Grushin operator $G_{\varkappa}$ is given by 
\begin{align} \label{def:grushin-control-metric}
\tilde{d} (x,y) = \sup_{\psi \in \mathcal{E}} \left| \psi(x) - \psi(y) \right|, 
\end{align}
where $\mathcal{E} = \left\{ \psi \in W^{1, \infty} (\mathbb{R}^{n_1 + n_2}) : \sum\limits_{1 \leq j \leq n_1} \left| X_{j} \psi \right|^2 + \sum\limits_{|\alpha| = \varkappa} \sum\limits_{1 \leq k \leq n_2} \left| X_{\alpha, k} \psi\right|^2 \leq 1 \right\}.$

It was established in Proposition $5.1$ of \cite{RobinsonSikoraDegenerateEllipticOperatorsGrushinTypeMathZ2008} that $\tilde{d}$ has the following asymptotics: 
\begin{align} \label{def:grushin-metric-asymp}
\tilde{d} (x,y) \sim d(x,y) := \left|x^{\prime} - y^\prime \right| + 
\begin{cases}
\frac{\left|x^{\prime \prime} - y^{\prime \prime}\right|}{\left(\left|x^{\prime} \right| + \left|y^\prime \right|\right)^{\varkappa}} &\textup{ if } \left|x^{\prime \prime} - y^{\prime \prime}\right|^{1/(1+\varkappa)} \leq \left|x^{\prime} \right| + \left|y^\prime \right| \\
\left|x^{\prime \prime} - y^{\prime \prime}\right|^{1/(1+\varkappa)} &\textup{ if } \left|x^{\prime \prime} - y^{\prime \prime}\right|^{1/(1+\varkappa)} \geq \left|x^{\prime} \right| + \left|y^\prime \right|. 
\end{cases}
\end{align}

Since $\tilde{d}$ is a metric, it follows that $d$ is a quasi-metric, that is, there exists a constant $C_0 > 1$ such that 
\begin{align} \label{def:quasi-constant}
d(x,y) \leq C_0 \left( d(x,z) + d(z,y) \right)    
\end{align}
for all $x, y, z \in \mathbb{R}^{n_1+n_2}$. 

Using the asymptotic estimates from \eqref{def:grushin-metric-asymp}, it can be easily verified that $ \left(\mathbb{R}^{n_1+n_2}, d \right)$ is a complete metric space (equivalently, $ (\mathbb{R}^{n_1+n_2}, \tilde{d} )$ is a complete quasi metric space).

By abuse of notation, throughout this article, we refer to $d(x,y)$ of \eqref{def:grushin-metric-asymp} as the Grushin metric. Let $B(x, r) := \{y \in \mathbb{R}^{n_1+n_2} : d(x, y)< r\}$ and $|B(x, r)|$ denote the Lebesgue measure of the ball $B(x, r)$ in $\mathbb{R}^{n_1 + n_2}$. Again, it was shown in Proposition $5.1$ of \cite{RobinsonSikoraDegenerateEllipticOperatorsGrushinTypeMathZ2008} that 
\begin{align} \label{grushin-ball-growth} 
\left|B(x,r)\right| \sim r^{n_1 + n_2} \max\{r, |x^{\prime}|\}^{\varkappa \, n_2},
\end{align}
for all $x\in \mathbb{R}^{n_1+n_2}$ and $r>0$. The ball volume estimates of \eqref{grushin-ball-growth} imply that the space $(\mathbb{R}^{n_1+n_2}, d, |\cdot|)$ is a homogeneous type metric space, that is, the underlying measure (the Lebesgue measure of $\mathbb{R}^{n_1 + n_2}$) satisfies the following doubling condition 
\begin{align} \label{ineq:doubling-condition-grushin-metric}
|B(x, sr)| \lesssim_{n_1, n_2, \varkappa} (1+s)^{Q} |B(x,r)|
\end{align}
for $s>0$, where 
$Q$ stands for the homogeneous dimension $n_{1}+(1+\varkappa) n_{2}$. 

Furthermore, since $(\mathbb{R}^{n_1+n_2}, d, |\cdot|)$ and $(\mathbb{R}^{n_1+n_2}, \tilde{d}, |\cdot|)$ are doubling metric measure spaces, it follows that open balls (of arbitrary finite radius) in any of these spaces are totally bounded (see Corollary 2.3 of \cite{Soria-Tradacete-doubling-constant-2019}). As a consequence, we get that the closed balls (of arbitrary finite radius) in $(\mathbb{R}^{n_1+n_2}, d, |\cdot|)$ or $(\mathbb{R}^{n_1+n_2}, \tilde{d}, |\cdot|)$ are compact. 

From geometric point of view, space $(\mathbb{R}^{n_1+n_2}, \tilde{d})$ carries a sub-Riemannian structure $g$ induced by the orthonormal frame $\{X_{j}, X_{\alpha, k} : 1 \leq j \leq n_1$, $1 \leq k \leq n_2$, and $\alpha \in \mathbb{N}^{n_1}$ with $|\alpha| = \varkappa \}$, where $X_{j}$ and $X_{\alpha, k}$ are given by \eqref{first-order-grad}. It is easy to see that these vector fields satisfy H\"ormander's hypoellipticity condition, and therefore it follows from Lemma $3.2$ and Theorem $7.1$ of \cite{Strichartz1986-Sub-Riemann} that given any $x, y \in \mathbb{R}^{n_1+n_2}$, there always exists a length minimising curve (geodesic in the sub-Riemannian set-up) joining $x$ and $y$. Moreover, it was shown in Section 3 of \cite{Jerison-Sanchez-Calle-85} that the distance $\tilde{d}(x,y)$ defined by \eqref{def:grushin-control-metric} equals the infimum of the length minimising curves joining $x$ and $y$. With that, we can invoke the proof of Lemma 3.2 of \cite{Strichartz1986-Sub-Riemann} to argue that given $x \in \mathbb{R}^{n_1+n_2}$ and $r>0$, if $\tilde{d}(x,y) < r$ then the length minimising curve $\gamma_0$ between $x$ and $y$ lies in the ball $\tilde{B}(x, 2r) := \{z \in \mathbb{R}^{n_1+n_2} : \tilde{d} (x, z)< 2r\}$. Moreover, 
\begin{align} \label{est:mean-value}
\left| f(y)-f(x)\right| = \left|\int_{0}^{1} \left(f\circ \gamma_0 \right)'(t) \, dt  \right|& \leq \int_{0}^{1} \left| \left(f\circ \gamma_0 \right)'(t) \right| dt  \\
\nonumber & \leq \int_{0}^{1} \left| df_{\gamma_0(t)} (\gamma_0' (t)) \right| dt \\ 
\nonumber & \leq \int_{0}^{1} \left\| \left( \nabla_g f \right) (\gamma_0 (t)) \right\|_g \left\| \gamma_0' (t) \right\|_g dt \\
\nonumber & = \tilde{d}(x,y) \int_{0}^{1} \left\| \left( \nabla_g f \right) (\gamma_0 (t)) \right\|_g dt \\ 
\nonumber & = \tilde{d}(x,y) \int_{0}^{1} \left|Xf(\gamma_0(t)) \right| \, dt.
\end{align}

Since $d \sim \tilde{d}$, we can use estimate \eqref{est:mean-value} to write an analogous estimate with respect to $d$. 
\begin{lemma}[Mean-value estimate]  \label{lem:general-grushin-Mean-value}
There exist constants $C_{1, \varkappa}, C_{2, \varkappa} > 0$ (depending also on $n_1$ and $n_2$) such that for any ball $B(x_0, r)$ and points $x, y \in B(x_0, r)$, there exists a $\tilde{d}$-length minimising curve $\gamma_0 : [0, 1] \to B(x_0, C_{1, \varkappa} \, r)$ joining $x$ to $y$, and for any $f \in C^1 \left( B(x_0, C_{1, \varkappa} \, r) \right)$,  
\begin{align} \label{est:general-grushin-Mean-value} 
\left| f(x) - f(y) \right| \leq & C_{2, \varkappa} \, d(x, y) \int_{0}^{1} \left|Xf(\gamma_0(t)) \right| \, dt. 
\end{align}
\end{lemma} 

%%%%%%%%%%%%%%%%%%%%%%%%%%%%%%%%%
%%%%%%%%%%%%%%%%%%%%%%%%%%%%%%%%%
%%%%%%%%%%%%%%%%%%%%%%%%%%%%%%%%%
%%%%%%%%%%%%%%%%%%%%%%%%%%%%%%%%%
%%%%%%%%%%%%%%%%%%%%%%%%%%%%%%%%%
%%%%%%%%%%%%%%%%%%%%%%%%%%%%%%%%%
%%%%%%%%%%%%%%%%%%%%%%%%%%%%%%%%%
%%%%%%%%%%%%%%%%%%%%%%%%%%%%%%%%%
%%%%%%%%%%%%%%%%%%%%%%%%%%%%%%%%%

\subsection{Weighted Plancherel estimates} 
\label{subsec:known-weighted-Plancherel}
Spectral multipliers corresponding to a self-adjoint positive definite operator (with appropriate heat kernel bounds) have a long history. In particular, we refer to \cite{DuongOuhabazSikoraWeightedPlancherel2002JFA} where the authors studied spectral multipliers on general homogeneous type metric spaces. For our convenience, we simply recall it for Grushin operators $G_{\varkappa}$. Let $p_t$ (for $t>0$) denote the heat kernel associated with the operator $G_{\varkappa}$, that is, $p_t(x,y)$ is the integral kernel of the operator $e^{-tG_{\varkappa}}$. It is well known (see, Corollary 6.6 of \cite{RobinsonSikoraDegenerateEllipticOperatorsGrushinTypeMathZ2008}) that $p_t$ satisfies the following Gaussian type bounds: there exist constants $b, C > 0$ such that 
\begin{align} \label{est:heat-kernel-pointwise-Grushin} 
\left| p_{t}(x,y) \right| \leq C |B(x,t^{1/2})|^{-1} \exp \left( -b   d(x,y)^{2} / t \right). 
\end{align}

It was shown in \cite{DuongOuhabazSikoraWeightedPlancherel2002JFA} that as a consequence of estimate \eqref{est:heat-kernel-pointwise-Grushin}, one has the following Plancherel estimates for kernels corresponding to symbols with compact support. Let $K_{m (G_{\varkappa})}$ denotes the integral kernel of the Grushin multiplier operator $m (G_{\varkappa})$, then 
\begin{align} \label{est:UnweightedPlancherelGrushin} 
\| K_{m (G_{\varkappa})}(x, \cdot)\|_{2} \lesssim |B(x, R^{-1})|^{-1/2} \|m\|_{L^{\infty}}, 
\end{align} 
for every bounded Borel function $m$ supported on $[0, R^2]$ for any $R>0$. 

Building on estimate \eqref{est:UnweightedPlancherelGrushin}, weighted Plancherel estimates of the following type were established in \cite{DuongOuhabazSikoraWeightedPlancherel2002JFA}: 
\begin{lemma}[Lemma 4.3, \cite{DuongOuhabazSikoraWeightedPlancherel2002JFA}] \label{lem:L2-weighted-Plancherel-estimate-grushin} 
For every $\mathfrak{r}, \epsilon >0$, we have 
\begin{align*} 
\left\| (1+Rd(x,\cdot))^\mathfrak{r} K_{m(G_{\varkappa})} (x, \cdot) \right\|_2 \lesssim_{\mathfrak{r}, \epsilon} |B(x,R^{-1})|^{-1/2} \left\| m ({R^2}~ \cdot ) \right\|_{W_{\mathfrak{r} + \epsilon}^{\infty}},
\end{align*} 
for every bounded Borel function $m$ supported on $[0, R^2]$ for any $R>0$. 
\end{lemma} 

Following the ideas of \cite{DuongOuhabazSikoraWeightedPlancherel2002JFA}, the authors of \cite{AnhBuiDuongSpectralMultipliersBesovTriebelLizorkin} recently studied pointwise weighted kernel estimates, establishing the following analogue of the above lemma: 
\begin{lemma}[Lemma 4.3, \cite{AnhBuiDuongSpectralMultipliersBesovTriebelLizorkin}] \label{lem:pointwise-weighted-Plancherel-estimate-grushin} 
For every $\mathfrak{r} , \epsilon >0$, we have 
\begin{align*} 
(1 + R d(x, y))^\mathfrak{r}  K_{m(G_{\varkappa})}(x, y) \lesssim_{\mathfrak{r} , \epsilon} |B(x, R^{-1})|^{-1/2} |B(y, R^{-1})|^{-1/2} \left\| m ({R^2}~ \cdot ) \right\|_{W_{\mathfrak{r}  + \epsilon}^{\infty}}, 
\end{align*} 
for every bounded Borel function $m$ supported on $[0, R^2]$ for any $R>0$. 
\end{lemma} 

Interpolating estimates of Lemmas \ref{lem:L2-weighted-Plancherel-estimate-grushin} and \ref{lem:pointwise-weighted-Plancherel-estimate-grushin}, one can conclude that given $p \in [2, \infty]$ and $\mathfrak{r}, \epsilon > 0$, 
\begin{align*} 
\left\| |B(\cdot, R^{-1})|^{1/2 - 1/p} \, (1+Rd(x,\cdot))^\mathfrak{r} K_{m (G_\varkappa)}(x, \cdot) \right\|_p \lesssim_{\mathfrak{r}, \epsilon, p} |B(x,R^{-1})|^{-1/2} \left\| m ({R^2} \cdot ) \right\|_{W_{\mathfrak{r}  + \epsilon}^{\infty}}, 
\end{align*} 
for every bounded Borel function $m$ supported on $[0, R^2]$ for any $R>0$. 

In \cite{DziubanskiSikoraLieApproach}, the authors systematically explained that given an operator of the type $G_{\varkappa}$, one can associate a nilpotent Lie group with it. In fact, the method of \cite{DziubanskiSikoraLieApproach} is applicable for a larger class of operators which include $G_{\varkappa}$. Now, the estimates for the gradients of the heat kernel of such Lie groups are well known (see, for example, \cite{VaropoulosSaloffCosteCoulhonAnalysisGeometryGroupsBook92}), and with that one can deduce analogous estimates for the gradients of the heat kernel of Grushin operator $G_{\varkappa}$. See, for example, Corollary 4.1 in \cite{DziubanskiSikoraLieApproach}, where the first order gradient estimate for the Poisson kernel is deduced. It is explicit from the methodology that the same could be done for the heat kernel and its gradients (of arbitrary order). In particular, one can show that the following gradient estimates hold true for every $\Gamma \in \mathbb{N}^{n_0}$: 
\begin{align} \label{est:derivative-bounds-heat-kernel-grushin} 
\left| \left(X^{\Gamma}_x p_{t} (x, y) \right) (x) \right| & \lesssim_{\Gamma} t^{-|\Gamma|/2} \left| B \left(x, t^{1/2} \right) \right|^{-1} \exp \left(-b d(x,y)^2 / t \right), \\ 
\nonumber \text{and} \left| \left(X^{\Gamma}_y p_{t} (x, y) \right) (y) \right| & \lesssim_{\Gamma} t^{-|\Gamma|/2} \left| B \left(x, t^{1/2} \right) \right|^{-1} \exp \left(-b d(x,y)^2 / t  \right). 
\end{align}

Again, with the help of the above mentioned transference technique, the authors of \cite{RobinsonSikoraMathZ2016} studied boundedness of Riesz transforms $X^{\Gamma} G_{\varkappa}^{-|\Gamma|/2}$ (of arbitrary order $\Gamma$) associated with the Grushin operator $G_{\varkappa} = - \Delta_{x^{\prime}} - |x^{\prime}|^{2 \varkappa} \Delta_{x^{\prime \prime}}$ (see Theorem 3.1 in \cite{RobinsonSikoraMathZ2016}). It is not difficult to verify that their technique is also applicable for the operator $G_{\varkappa} = - \Delta_{x^{\prime}} - \sum_{j = 1}^{n_1} {x_j^{\prime}}^{2 \varkappa} \Delta_{x^{\prime \prime}}$. 

Summarising, with $G_{\varkappa} = - \Delta_{x^{\prime}} - |x^{\prime}|^{2 \varkappa} \Delta_{x^{\prime \prime}}$ or $G_{\varkappa} = - \Delta_{x^{\prime}} - \sum_{j = 1}^{n_1} {x_j^{\prime}}^{2 \varkappa} \Delta_{x^{\prime \prime}}$, we have 
\begin{align} \label{est:L2-boundedness-Riesz-transforms}
\| X^{\Gamma} G_{\varkappa}^{-|\Gamma|/2} \|_{L^p \left( \mathbb{R}^{n_1+n_2} \right) \to L^p \left( \mathbb{R}^{n_1+n_2} \right)} < \infty
\end{align}
for any $1 < p < \infty$, and also $\| X^{\Gamma} G_{\varkappa}^{-|\Gamma|/2} \|_{L^1 \left( \mathbb{R}^{n_1+n_2} \right) \to L^{1, \infty} \left( \mathbb{R}^{n_1+n_2} \right)} < \infty$. 

%%%%%%%%%%%%%%%%%%%%%%%%%%%%%%%%%
%%%%%%%%%%%%%%%%%%%%%%%%%%%%%%%%%
%%%%%%%%%%%%%%%%%%%%%%%%%%%%%%%%%
%%%%%%%%%%%%%%%%%%%%%%%%%%%%%%%%%
%%%%%%%%%%%%%%%%%%%%%%%%%%%%%%%%%
%%%%%%%%%%%%%%%%%%%%%%%%%%%%%%%%%
%%%%%%%%%%%%%%%%%%%%%%%%%%%%%%%%%
%%%%%%%%%%%%%%%%%%%%%%%%%%%%%%%%%
%%%%%%%%%%%%%%%%%%%%%%%%%%%%%%%%%

\subsection{Homogeneous spaces and \texorpdfstring{$A_p$}{} weights} 
Starting with the pioneering work of Muckenhoupt in \cite{Muckenhoupt-TAMS1972}, the theory of $A_p$ weights has been studied in various aspects. In recent times  obtaining quantitative weighted estimates for operators in Harmonic analysis is an active area of research. The development of the seminal $A_2$-conjecture gave rise to the powerful technique of sparse operators and it has been studied in various contexts even beyond Euclidean spaces. We mention a few here. 

Recall that a metric (or quasi-metric) measure space is of homogeneous type in the sense of Coifman--Weiss \cite{Coifman-Weiss-book-1971}) if the underlying measure satisfies the doubling condition. We have already seen that the space $(\mathbb{R}^{n_1 + n_2}, d, |\cdot|)$ is of homogeneous type. The theory of Calder\'on-Zygmund operators and $A_p$ weights have been studied systematically in the spaces of homogeneous type. The $A_2$ conjecture for the spaces of homogeneous type was established in \cite{Anderson-Vagharshakyan-simple-proof-JGA2014, Nazarov-Reznikov-Volberg-A2-proof-Indiana2013}. In \cite{Hytonen-Perez-Rela-reverse-Holder-JFA2012}, Hyt\"onen--P\'erez--Rela developed sharp reverse H\"older's inequality in spaces of homogeneous type. In recent times, the theory of sparse operators is in fact extended to spaces that need not be of homogeneous type, we refer interested readers to \cite{volberg2018sparse,conde2019nondoubling}. 

Let us recall $A_p$ weights on the homogeneous type space $(\mathbb{R}^{n_1 + n_2}, d, |\cdot|)$. 
\begin{definition} 
For $1 < p <\infty$, we say a weight $w$ on $\mathbb{R}^{n_1+n_2}$ belongs to the class $A_{p}(\mathbb{R}^{n_1+n_2})$ if 
\begin{align*}
[w]_{A_p(\mathbb{R}^{n_1+n_2})} = \sup_{B} \left( \frac{1}{|B|} \int_{B}w~ dx\right) \left( \frac{1}{|B|} \int_{B} w^{1-p^{\prime}} dx \right)^{p-1} <\infty, 
\end{align*} 
where the supremum is taken over all balls $B$ in  the homogeneous type space $(\mathbb{R}^{n_1+n_2}, d, |\cdot|)$. For $p=1$, $w\in A_{1}(\X)$ if there exists a constant $C>0$ such that for all balls $B$, 
\begin{align*}
\frac{1}{|B|} \int_{B}w~ dx \leq C \, \essinf_{x\in B}w(x). 
\end{align*}
\end{definition}

It is not difficult to construct examples of $A_p$ weights in our case. In fact, power weights provide plenty of examples in the case of the Euclidean $A_p$ weights.

\begin{example}
Consider the weights $\omega_{a}(x)=d(0, x)^{-a}$ for $a\in \mathbb{R}$. Then $\omega_{a}\in A_{1}$ if $0\leq a< Q$. To see this, consider any ball $B=B(x_{0}, r)$. If $B\cap B(0, \frac{r}{2})=\emptyset$ then for any $x, y\in B$, $d(0, x)\leq C_{0}(d(x, y)+d(y, 0))\lesssim  C_{0}^2 (r+d(y, 0))\lesssim C_{0}^2 d(y, 0)$, hence
$$\frac{1}{|B|}\int_{B}d(0, x)^{-a}\, dx \lesssim ess\inf_{x\in B} d(x,0)^{-a}.$$
For balls $B$ such that $B\cap B(0, \frac{r}{2})\neq \emptyset$ we proceed as follows. Observe that in this case $B\subset B(0, 2C_{0} r)$ which in turn implies $B(0, 2C_{0} r)\subset B_{\mathbb{R}^{n_1}}(0, c r)\times B_{\mathbb{R}^{n_2}}(0,c r^{(1+\varkappa)})$, where $B_{\mathbb{R}^{n_1}}(0, cr)$ and $B_{\mathbb{R}^{n_2}}(0,c r^{(1+\varkappa)})$ are Euclidean balls in $\mathbb{R}^{n_1}$ and $\mathbb{R}^{n_2}$ respectively. Now
\begin{align*}
\frac{1}{|B|}\int_{B}d(0, x)^{-a}\, dx &\leq \frac{1}{r^{n_1+(1+\varkappa) n_2}}\int_{B_{\mathbb{R}^{n_1}}(0, cr)}\int_{B_{\mathbb{R}^{n_2}}(0,c r^{(1+\varkappa)})}\max\{|x'|, |x''|^{1/(1+\varkappa)}\}^{-a} \ dx'\ dx''\\
&\lesssim \frac{1}{r^{n_1+(1+\varkappa)n_2}}r^{-a/2+n_{1}}r^{-a/2+(1+\varkappa)n_{2}}\lesssim r^{-a} \lesssim ess\inf_{x\in B} d(0, x)^{-a}.
\end{align*}
\end{example}

Since, $v_{0}v_{1}^{1-p}\in A_{p}$ for $v_{i}\in A_{1},$ we can conclude that for $1<p<\infty$, $\omega_{a}(x)=d(0, x)^{-a}\in A_{p}(\X)$ if $-Q(p-1)<a< Q$.

%%%%%%%%%%%%%%%%%%%%%%%%%%%%%%%%%
%%%%%%%%%%%%%%%%%%%%%%%%%%%%%%%%%
%%%%%%%%%%%%%%%%%%%%%%%%%%%%%%%%%
%%%%%%%%%%%%%%%%%%%%%%%%%%%%%%%%%
%%%%%%%%%%%%%%%%%%%%%%%%%%%%%%%%%
%%%%%%%%%%%%%%%%%%%%%%%%%%%%%%%%%
%%%%%%%%%%%%%%%%%%%%%%%%%%%%%%%%%
%%%%%%%%%%%%%%%%%%%%%%%%%%%%%%%%%
%%%%%%%%%%%%%%%%%%%%%%%%%%%%%%%%%

\subsection{Sparse domination} \label{subsec:sparse-domination}
In this article, we address quantitative weighted estimates for Grushin pseudo-multiplier operators with respect to $A_p(\mathbb{R}^{n_1+n_2})$ weights. In order to state the results, let us first recall the definition of sparse families and sparse operators on a homogeneous type space. We also need the following notion of dyadic cubes for a homogeneous type space. For details, we refer to \cite{Christ-lectures-singular-1990,Lorist-pointwaise-sparse2021}. 

\begin{theorem}[\cite{Christ-lectures-singular-1990}]
\label{thm:dyadic-cubes}
There exist absolute constants $0 < h_1 \leq h_2 < \infty, \, 0 < \varepsilon_0, \delta_0 < 1,$ a family of measurable sets $\mathcal{S} = \bigcup_{k\in \mathbb{Z}} \mathcal{S}_k$ (called a dyadic decomposition of $(\mathbb{R}^{n_1+n_2}, d, | \cdot |)$) and a corresponding family of points $\{ c_{\mathcal{Q}} \}_{\mathcal{Q} \in \mathcal{S}}$ that satisfy the following properties:
\begin{enumerate}[(i)]
\item $\left| \mathbb{R}^{n_1+n_2} \setminus \bigcup\limits_{\mathcal{Q} \in \mathcal{S}_{k}} \mathcal{Q} \right| = 0$ \, for all $k\in \mathbb{Z}$.

\item For $k \geq l$, if $\mathcal{P} \in \mathcal{S}_k$ and $\mathcal{Q} \in \mathcal{S}_l$, then either $\mathcal{P} \cap \mathcal{Q}=\emptyset$ or $\mathcal{P} \subset \mathcal{Q}$. 

\item For any $\mathcal{Q}_1 \in \mathcal{S}_{k}$ there exists at least one $\mathcal{Q}_2 \in  \mathcal{S}_{k+1}$ such that $\mathcal{Q}_2 \subset \mathcal{Q}_1$ (called a child of $\mathcal{Q}_1$)
and there exists exactly one $\mathcal{Q}_3 \in \mathcal{S}_{k-1}$ such that $\mathcal{Q}_1 \subset \mathcal{Q}_3$ (called the parent of $\mathcal{Q}_1$).

\item If $\mathcal{Q}_2$ is a child of $\mathcal{Q}_1$ then $|\mathcal{Q}_2| \geq \varepsilon_0 |\mathcal{Q}_1|$. 

\item For every $\mathcal{Q} \in \mathcal{S}_k$ we have 
$$ B(c_{\mathcal{Q}}, h_1 \delta_0^k) \subset \mathcal{Q} \subset B(c_{\mathcal{Q}}, h_2 \delta_0^k), $$
and for any $s>0$, we denote $B(c_{\mathcal{Q}}, s h_2 \delta_0^k)$ by $s B(\mathcal{Q})$. 
\end{enumerate}
\end{theorem}

Recall the definition of the sparse family $S \subset \mathcal{S}$ and the sparse operators $\mathcal{A}_{r, S}$ and $\mathcal{A}_{S}$ given in and around \eqref{def:Sparse-operator} in the introduction. 

We let $\mathcal{M}$ stand for the uncentered Hardy--Littlewood maximal function on $(\mathbb{R}^{n_1+n_2}, d, | \cdot |)$ and for $1<r<\infty$, $\mathcal{M}_{r}f:=(\mathcal{M}(|f|^r))^{1/r}$. Our proofs on sparse domination depend on a general sparse domination principle by Lerner--Ombrosi (Theorem 1.1 in \cite{Lerner-Ombrosi-pointwaise-sparse2020}) which is recently extended to the homogeneous type spaces by Lorist \cite{Lorist-pointwaise-sparse2021}. Corresponding to the dyadic decomposition $\mathcal{S}$ and the absolute constants $ 0 < h_1 \leq h_2 < \infty$, $0 < \varepsilon_0, \delta_0 < 1$ (as in Theorem \ref{thm:dyadic-cubes}), and for any linear operator $T$ acting on functions defined on $ \mathbb{R}^{n_1+n_2}$, let us consider the following version of the grand maximal truncated operator: For $s > 0$
\begin{align} \label{def:grand-maximal-truncated-operator}
\mathcal{M}^{\#}_{T, s} f(x) = \sup_{B: B \ni x} \esssup \limits_{y,z \in B} \left| T(f \chi_{\mathbb{R}^{n_1+n_2}\setminus s B})(y)-T(f \chi_{\mathbb{R}^{n_1+n_2}\setminus s B})(z) \right|, 
\end{align}
where the supremum is taken over all balls $B$ containing the point $x$.

With the notations as defined above and with $C_0$ as in \eqref{def:quasi-constant}, the following result was proved in \cite{Lorist-pointwaise-sparse2021}. For our convenience, we state it for the homogeneous type space $(\mathbb{R}^{n_1+n_2}, d, | \cdot |)$.

\begin{theorem}[Theorem 1.1, \cite{Lorist-pointwaise-sparse2021}]
\label{thm:Lorist-general-sparse-principle}
On the homogeneous type space $(\mathbb{R}^{n_1+n_2}, d, | \cdot |)$, let $T$ be a sub-linear operator of weak type $(p, p)$ and $\mathcal{M}^{\#}_{T, s}$ is weak type $(q, q)$ for some $1\leq  p, q<\infty$ and $s\geq \frac{3 C_{0}^{2}}{\delta_0}$. Let $r=\max\{p, q\}$. Then there is an $0 < \eta < 1$ such that for every compactly supported bounded measurable function $f$, there exist an $\eta$-sparse family $S \subset \mathcal{S}$ such that for almost every $x \in \mathbb{R}^{n_1 + n_2}$ we have 
\begin{align*}
|Tf(x)| \lesssim_{\eta, s} C_T \mathcal{A}_{r, S}f(x), 
\end{align*} 
where $C_T = \|T\|_{L^p \to L^{p, \infty}} + \| \mathcal{M}^{\#}_{T, s} \|_{L^q \to L^{q, \infty}}$. 
\end{theorem}

The primary usefulness of the sparse domination lies in the fact that it is easy to get quantitative weighted estimates for sparse operators which can be immediately passed on to the operator at hand. Quantitative weighted bounds for the sparse operators are rigorously studied in Euclidean spaces and almost analogous results are also established in homogeneous type spaces as well. In particular, we state the following result from \cite{Lorist-pointwaise-sparse2021}. 
\begin{theorem}[Proposition 4.1, \cite{Lorist-pointwaise-sparse2021}]
\label{thm:Quantitative-bounds-sparse-operators}
Let $S$ be an $\eta$-sparse family and $r\in [1, \infty)$. Then for $p\in (r, \infty)$, $w\in A_{p/r}(\X)$ and $f\in L^p(\X, w)$ we have
$$\|\mathcal{A}_{r, S}f\|_{L^p(w)}\lesssim [w]_{A_{p/r}(\X)}^{\max\{\frac{1}{p-r}, 1\}}\|f\|_{L^p(w)},$$
with the implicit constant depending on $p, r$, and $\eta$. 
\end{theorem}

In the follow up, we shall confine ourselves only to sparse domination results for operators at hand. Since the associated quantitative weighted estimates follow from Theorem \ref{thm:Quantitative-bounds-sparse-operators}, for brevity, we refrain ourselves from stating them each and every time. 

%%%%%%%%%%%%%%%%%%%%%%%%%%%%%%%%%%%%%%%%%%%%%
%%%%%%%%%%%%%%%%%%%%%%%%%%%%%%%%%%%%%%%%%%%%%
%%%%%%%%%%%%%%%%%%%%%%%%%%%%%%%%%%%%%%%%%%%%%
%%%%%%%%%%%%%%%%%%%%%%%%%%%%%%%%%%%%%%%%%%%%%
%%%%%%%%%%%%%%%%%%%%%%%%%%%%%%%%%%%%%%%%%%%%%
%%%%%%%%%%%%%%%%%%%%%%%%%%%%%%%%%%%%%%%%%%%%%
%%%%%%%%%%%%%%%%%%%%%%%%%%%%%%%%%%%%%%%%%%%%%
%%%%%%%%%%%%%%%%%%%%%%%%%%%%%%%%%%%%%%%%%%%%%
%%%%%%%%%%%%%%%%%%%%%%%%%%%%%%%%%%%%%%%%%%%%%
%%%%%%%%%%%%%%%%%%%%%%%%%%%%%%%%%%%%%%%%%%%%%
%%%%%%%%%%%%%%%%%%%%%%%%%%%%%%%%%%%%%%%%%%%%%
%%%%%%%%%%%%%%%%%%%%%%%%%%%%%%%%%%%%%%%%%%%%%
%%%%%%%%%%%%%%%%%%%%%%%%%%%%%%%%%%%%%%%%%%%%%
%%%%%%%%%%%%%%%%%%%%%%%%%%%%%%%%%%%%%%%%%%%%%

\section{Kernel estimates and pointwise sparse domination} \label{sec:kernel-estimates-and-pointwise-sparse-domination}

We dedicate this section to obtain a framework for sparse domination for a class of operators $T \in \mathcal{B} \left( L^2({\mathbb{R}^{n_1+n_2}}) \right)$ satisfying certain properties that we shall shortly discuss. 

Let $T \in \mathcal{B} \left( L^2({\mathbb{R}^{n_1+n_2}}) \right)$ be such that it admits a decomposition $T=\sum_{j\geq 0}T_{j}$, with the convergence in the strong operator topology. We also assume that each operator $T_j$ has an integral kernel, denoted by $T_{j}(x, y).$ In this framework, we shall be working with conditions on the kernels of the following type, with the gradient vector $X$ as in \eqref{first-order-grad-vector}.

\medskip 
\noindent \textbf{\underline{$L^2$-conditions on the kernel}:} There exists some $R_{0} \in (0, \infty)$ such that for all $ \mathfrak{r} \in [0, R_{0}]$, $j \geq 0$, and for every positive real number $\mathcal{K}_0,$ we have 
\begin{align} 
\sup_{x\in \mathbb{R}^{n_1+n_2}} |B(x, 2^{-j/2})| \int_{\mathbb{R}^{n_1+n_2}} d(x,y)^{2\mathfrak{r}} |T_{j}(x,y)|^2 \, dy & \lesssim_{R_0} 2^{-j\mathfrak{r}}, \label{cond:General-hypo} \\ 
\sup_{x\in \mathbb{R}^{n_1+n_2}} |B(x, 2^{-j/2})| \int_{d(x,y)<\mathcal{K}_0}  d(x,y)^{2\mathfrak{r}} |X_{x} T_{j}(x,y)|^2 \, dy & \lesssim_{R_0, \mathcal{K}_0} 2^{-j \mathfrak{r}} 2^{j} \label{cond:General-hypo-grad}. 
\end{align}

\noindent \textbf{\underline{$L^{\infty}$-conditions on the kernel}:} There exists some $R_{0} \in (0, \infty)$ such that for all $\mathfrak{r} \in [0, R_{0}]$, $j \geq 0$, and for every positive real number $\mathcal{K}_0$, we have 
\begin{align} 
\sup_{x\in \mathbb{R}^{n_1+n_2}} \sup_{y \in \mathbb{R}^{n_1+n_2}} |B(x, 2^{-j/2})|^{1/2} |B(y, 2^{-j/2})|^{1/2} d(x,y)^{\mathfrak{r}} |T_{j}(x,y)| & \lesssim_{R_0} 2^{-j\mathfrak{r}/2}, \label{cond:General-hypo-sup} \\ 
\sup_{x \in \mathbb{R}^{n_1+n_2}} \sup_{y \in \mathbb{R}^{n_1+n_2}} |B(x, 2^{-j/2})|^{1/2} |B(y, 2^{-j/2})|^{1/2} d(x,y)^{\mathfrak{r}} |X_{y}T_{j}(x,y)| & \lesssim_{R_0} 2^{-j\mathfrak{r}/2} 2^{j/2}. \label{cond:General-hypo-y-grad-sup} \\ 
\sup_{x \in \mathbb{R}^{n_1+n_2}} \sup_{d(x,y)<\mathcal{K}_0} |B(x, 2^{-j/2})|^{1/2} |B(y, 2^{-j/2})|^{1/2} d(x,y)^{\mathfrak{r}} |X_{x}T_{j}(x,y)| & \lesssim_{R_0, \mathcal{K}_0} 2^{-j\mathfrak{r}/2} 2^{j/2}, \label{cond:General-hypo-grad-sup} 
\end{align}

\begin{remark} \label{rem:explaining-conditions-general-operator}
Conditions \eqref{cond:General-hypo}, \eqref{cond:General-hypo-sup} and \eqref{cond:General-hypo-y-grad-sup} are motivated by weighted Plancherel estimates obtained for spectral multipliers in \cite{DuongOuhabazSikoraWeightedPlancherel2002JFA, AnhBuiDuongSpectralMultipliersBesovTriebelLizorkin}. In conditions \eqref{cond:General-hypo-grad} and \eqref{cond:General-hypo-grad-sup}, we restrict the integral (or supremum) only on compact sets. In fact, this helps us to achieve the sparse domination with minimal requirement of derivatives on the symbol function in Theorems \ref{thm:pseudo-grushin-a=0-less-derivative} and \ref{thm:pseudo-grushin-a=0-full-derivative}. A similar idea was already employed by Bagchi--Thangavelu in their study of Hermite pseudo-multipliers (see Proposition $4.3$ of \cite{BagchiThangaveluHermitePseudo}). 
\end{remark}

Following are our main pointwise sparse domination results for such class of operators.
\begin{theorem} \label{thm:main-sparse}
Let $T \in \mathcal{B} \left( L^2({\mathbb{R}^{n_1+n_2}}) \right)$ be such that it admits a decomposition $T=\sum_{j\geq 0}T_{j}$, with the convergence in the strong operator topology, and that the integral kernels $T_j(x,y)$ satisfy conditions \eqref{cond:General-hypo} and \eqref{cond:General-hypo-grad} for some $R_0 > Q/2$. Then for every compactly supported bounded measurable function $f$ there exists a sparse family $S\subset \mathcal{S}$ such that
\begin{align} \label{sparseHormander1-general} 
|Tf(x)| \lesssim_{T} \mathcal{A}_{2, S} f(x),
\end{align} 
for almost every $x \in \mathbb{R}^{n_1 + n_2}$. 
\end{theorem}

The following theorem reflects the well-known fact that if we assume appropriate pointwise estimates on the kernels $T_j(x, y)$ then we obtain sharper estimates for the operator $T$.
\begin{theorem} \label{thm:main-sparse-more-derivative}
Let $T \in \mathcal{B} \left( L^2({\mathbb{R}^{n_1+n_2}}) \right)$ be such that it admits a decomposition $T=\sum_{j\geq 0}T_{j}$, with the convergence in the strong operator topology, and that the integral kernels $T_j(x,y)$ satisfy conditions \eqref{cond:General-hypo-sup} and \eqref{cond:General-hypo-y-grad-sup} for some $R_0 \geq Q + \frac{1}{2}$, and condition \eqref{cond:General-hypo-grad-sup} for some $R_0 > Q$. Then for every compactly supported bounded measurable function $f$ there exists a sparse family $S\subset \mathcal{S}$ such that
\begin{align}
\label{sparseHormander2-general}
|Tf(x)| \lesssim_{T} \mathcal{A}_{S} f(x),
\end{align}
for almost every $x \in \mathbb{R}^{n_1+n_2}$. 
\end{theorem}

We develop proofs of Theorems \ref{thm:main-sparse} and \ref{thm:main-sparse-more-derivative} over the next two subsections. 

%%%%%%%%%%%%%%%%%%%%%%%%%%%%%%%%%
%%%%%%%%%%%%%%%%%%%%%%%%%%%%%%%%%
%%%%%%%%%%%%%%%%%%%%%%%%%%%%%%%%%
%%%%%%%%%%%%%%%%%%%%%%%%%%%%%%%%%
%%%%%%%%%%%%%%%%%%%%%%%%%%%%%%%%%
%%%%%%%%%%%%%%%%%%%%%%%%%%%%%%%%%
%%%%%%%%%%%%%%%%%%%%%%%%%%%%%%%%%
%%%%%%%%%%%%%%%%%%%%%%%%%%%%%%%%%
%%%%%%%%%%%%%%%%%%%%%%%%%%%%%%%%%

\subsection{Proof of Theorem \ref{thm:main-sparse}} \label{subsec:proof-thm:main-sparse}
In view of Theorem \ref{thm:Lorist-general-sparse-principle}, Theorem \ref{thm:main-sparse} follows from the following result. 

\begin{theorem} \label{thm:domination-maximal-M2}
Let $T \in \mathcal{B} \left( L^2({\mathbb{R}^{n_1+n_2}}) \right)$ be such that it admits a decomposition $T=\sum_{j\geq 0}T_{j}$, with the convergence in the strong operator topology, and that the integral kernels $T_j(x,y)$ satisfy conditions \eqref{cond:General-hypo} and \eqref{cond:General-hypo-grad} for some $R_0 > Q/2$. Then, for $s = 3 C_{1, \varkappa} C^2_{0} \delta_0^{-1}$ with $\delta_0$ as in Theorem \ref{thm:dyadic-cubes} and $C_{1, \varkappa}$ as in Lemma \ref{lem:general-grushin-Mean-value}, we have the following pointwise almost everywhere estimate for every bounded measurable function $f$ with compact support: 
\begin{align*} 
\mathcal{M}^{\#}_{T, s} f(x) \lesssim_{T, \varkappa, s} \mathcal{M}_{2}f(x). 
\end{align*} 
\end{theorem}
\begin{proof}
Choose and fix a small $\epsilon_1 >0$ such that $Q(1 + \epsilon_1) < 2 R_0$. Now, fix $x \in \mathbb{R}^{n_1+n_2}$ and a ball $B=B(z_0, r)$ containing $x$. Denote $r_{0}=sr$. Let $k_0$ be the smallest natural number such that $|B(z_0,2^{k_0} r_0)|>2|B(z_0, r_0)|$, and write $r_1 = 2^{k_0} r_{0}$. Next, let $k_1$ be the smallest natural number such that $|B( z_0,2^{k_1} r_1)| > 2|B( z_0, r_1)|$, and write $r_2 = 2^{k_1} r_{1}$. Continuing this process, for each $l \in \mathbb{N}$, let $k_{l}$ be the smallest natural number corresponding to $r_{l}$, that is, $r_{l+1} = 2^{k_{l}} r_{l}$ with $|B( z_0, r_{l+1})| > 2|B( z_0, r_{l})|$ and $|B( z_0, r_{l+1}/2)| = |B( z_0, 2^{k_{l}-1}r_{l})| \leq 2|B( z_0, r_{l})|$. As an immediate consequence of this construction, we get the following two estimates: 
\begin{align} 
& |B(x, r_{l})| \sim |B( z_0, r_{l})| \gtrsim |B( z_0, r_{l+1}/2)| \gtrsim |B( z_0, r_{l+1})| \sim |B(x, r_{l+1})|, \label{relation:annulus-part-ball-vol-comp-sparse-proof} \\ 
& |B(x, r_{l})| \gtrsim 2^l |B(x, r_0)| \gtrsim 2^{l} r_0^{n_1 + n_2} \max \left\{ r_0^{\varkappa n_2}, |x'|^{\varkappa n_2} \right\}. \label{relation:l-stage-ball-vol-comp-sparse-proof}
\end{align}
Also, denoting the annulus $B( z_0, r_{l+1})\setminus B( z_0, r_{l})$ by $\mathscr{A}_{l}$, it follows easily from the choice of $s = 3 C_{1, \varkappa} C^2_{0} \delta_0^{-1}$ that 
\begin{align} \label{relation:2-annulus-part-ball-vol-comp-sparse-proof}
d( z_0, v) \sim d(w, v),
\end{align}
for any $w \in B( z_0, C_{1, \varkappa} \, r)$ and $v \in \mathscr{A}_{l}$, with the implicit constants uniform in $l$ and $r$. 

We perform our analysis by decomposing operators $T_j$ as follows. Let $\phi\in C^{\infty}_{c}(\mathbb{R})$ be such that it is supported in $[-5C^2_0,5C^2_0]$ and $\phi \equiv 1$ on $[-3C^2_0, 3C^2_0]$. Let $T^1_j$ and $T^2_j$ be operators with kernels given by $T^1_j(u,v)= T_j(u, v) \phi(d(u, v)) $ and $T^2_j(u, v)= T_j(u, v) (1-\phi(d(u, v)) ) $ respectively. Benefit of introducing such a decomposition is explained later in Remark \ref{rem:introduction-phi-function-minimal-smoothness}.

For $y, z\in B( z_0, r)$, we write
\begin{align} \label{eq:M2-sparse-thm-estimate0} 
& |T(f\chi_{\X\setminus B( z_0, r_0}))(y)-T(f\chi_{\X\setminus B( z_0, r_0}))(z)| \\ 
\nonumber & \leq \sum_{j\geq 0}|T_{j}(f\chi_{\mathbb{R}^{n_1+n_2}\setminus B( z_0, r_{0})})(y)-T_{j}(f\chi_{\mathbb{R}^{n_1+n_2}\setminus B( z_0, r_{0})})(z)|\\
\nonumber & \leq \sum_{j\geq 0} |T^{1}_{j}(f\chi_{\mathbb{R}^{n_1+n_2}\setminus B( z_0, r_{0})})(y)-T^{1}_{j}(f\chi_{\mathbb{R}^{n_1+n_2}\setminus B( z_0, r_{0})})(z)|\\
\nonumber & \quad + \sum_{j\geq 0}|T^{2}_{j}(f\chi_{\mathbb{R}^{n_1+n_2}\setminus B( z_0, r_{0})})(y)-T^{2}_{j}(f\chi_{\mathbb{R}^{n_1+n_2}\setminus B( z_0, r_{0})})(z)| \\ 
\nonumber & =: \mathfrak{I}_{1} + \mathfrak{I}_{2}. 
\end{align} 

\medskip \noindent \textbf{\underline{Estimation of $\mathfrak{I}_{1}$ in \eqref{eq:M2-sparse-thm-estimate0}}:} 
Note that 
\begin{align} \label{eq:M2-sparse-thm-estimate1} 
& \mathfrak{I}_{1} \quad \leq \sum_{j \geq 0} \int_{\mathbb{R}^{n_1+n_2} \setminus B( z_0, r_0)} |T^{1}_{j}(y,v)-T^{1}_{j}(z,v)| \, |f(v)| \, dv =: \sum_{j\geq 0} I_{j},
\end{align}
and each $I_j$ could be further decomposed over annuli $\mathscr{A}_l$ as 
\begin{align} \label{eq:M2-sparse-thm-estimate1-annulusdecomposition-Case1} 
I_{j} \leq \sum_{l=0}^{\infty} \int_{\mathscr{A}_{l}}|T^{1}_{j}(y,v)-T^{1}_{j}(z,v)| \, |f(v)| \, dv. 
\end{align} 
We consider each term in the infinite sum in \eqref{eq:M2-sparse-thm-estimate1-annulusdecomposition-Case1}. Each such term is again dominated by $$\int_{\mathscr A_{l}}|T^{1}_{j}(y,v)| \, |f(v)|\,dv+\int_{\mathscr A_{l}}|T^{1}_{j}(z,v)| \, |f(v)| \, dv.$$
Since both the terms above are similar we just estimate one of them.

Depending on the nature of the metric $d,$ we make a further decomposition of the annulus $\mathscr{A}_{l}$ into the regions $\{v\in \mathscr A_{l}: |y'| \leq d( z_0,v)\}$ and $\{v\in \mathscr A_{l}: |y'| > d( z_0,v)\}$ yielding the following estimates:
\begin{align} \label{mainthm1}
&\int_{\mathscr A_{l} \cap \{|y'| \leq d( z_0, v)\}}|T^{1}_{j}(y,v)| \, |f(v)|\, dv\\
 \nonumber& \leq \left( \int_{\mathscr A_{l} \cap \{|y'| \leq d( z_0, v)\}} |B(y, d( z_0, v))|^{-(1 +\epsilon_1)} |f(v)|^2 \, dv \right)^{1/2} \\
\nonumber & \quad \times \left(\int_{\mathscr A_{l} \cap \{|y'| \leq d( z_0, v)\}} |B(y, d( z_0,v))|^{(1+\epsilon_1)} |T^{1}_{j}(y,v)|^{2}\, dv \right)^{1/2} \\
\nonumber & \lesssim \left( |B(y, r_{l})|^{-(1+\epsilon_1)} \int_{B( z_0, r_{l+1})} |f(v)|^2 \, dv  \right)^{1/2} \left( \int_{\mathscr A_{l}} d( z_0, v)^{Q (1 + \epsilon_1)} |T_{j}(y,v)|^{2}\, dv \right)^{1/2} \\ 
\nonumber & \lesssim \frac{1}{|B(y, r_{l})|^{\frac{\epsilon_1}{2}}} \left(\frac{1}{|B( z_0, r_{l+1})|} \int_{B( z_0, r_{l+1})}|f(v)|^2 \, dv \right)^{1/2} \left( \int_{\mathscr A_{l}} d(y,v)^{Q (1 + \epsilon_1)} |T_{j}(y,v)|^{2}\, dv \right)^{1/2} \\ 
\nonumber &\lesssim_{\epsilon_1} \left( 2^l \, r_0^Q \right)^{-\epsilon_1 / 2} \mathcal M_{2}f(x) \, 2^{-jQ (1 + \epsilon_1)/4} |B(y, 2^{-j/2})|^{-1/2} \\ 
\nonumber & \lesssim 2^{- l \epsilon_1 / 2} \left(2^{j/2} r_0 \right)^{-Q \epsilon_1 /2} \mathcal M_{2}f(x), 
\end{align} 
where we have used $|B(y, r_{l})| \gtrsim |B( z_0, r_{l+1})|$ from \eqref{relation:annulus-part-ball-vol-comp-sparse-proof} in the third from the last inequality, the ball volume lower bound from \eqref{relation:l-stage-ball-vol-comp-sparse-proof} and condition \eqref{cond:General-hypo} in the second last inequality.

Next, by suitably modifying the above calculations, we obtain 
\begin{align} \label{mainthm2} 
&\int_{\mathscr{A}_{l} \cap \{|y'|> d( z_0, v)\}}| T^{1}_{j}(y,v)| \, |f(v)|\, dv \\
\nonumber &\leq \left(\int_{\mathscr A_{l} \cap \{|y'|> d( z_0, v)\}}|B(y, d( z_0, v))|^{-(1+ \epsilon_1)}|f(v)|^2 \, dv \right)^{\frac{1}{2}}\\
\nonumber & \quad \times \left(\int_{\mathscr A_{l} \cap \{|y'|>d( z_0, v)\}}|B(y, d( z_0, v))|^{1+ \epsilon_1} |T^{1}_{j}(y,v)|^{2}\, dv \right)^{\frac{1}{2}}\\
\nonumber & \lesssim \left(\frac{1}{|B(y, r_{l})|^{ \epsilon_1} |B( z_0, r_{l+1})|} \int_{B( z_0, r_{l+1})}|f(v)|^2 \, dv \right)^{\frac{1}{2}} \\ 
\nonumber & \quad \times \left(\int_{\mathscr A_{l} \cap \{|y'|> d( z_0, v)\}} d( z_0, v)^{(n_1+ n_2)(1+ \epsilon_1)}|y'|^{\varkappa n_2(1+ \epsilon_1)}| T_{j}(y,v)|^{2} |\phi (y,v)|^{2} \, dv \right)^{\frac{1}{2}} \\
\nonumber & \lesssim_{\epsilon_1} 2^{-l \epsilon_1 / 2} r_{0}^{-(n_1+ n_2) \epsilon_1/2} |y'|^{\varkappa n_2 / 2} \mathcal M_2f(x) \, 2^{-j(n_1+ n_2) (1+\epsilon_1) / 4} |B(y, 2^{-j/2})|^{-1/2} \\
\nonumber & = 2^{-l \epsilon_1 / 2} \left( 2^{j/2} r_{0} \right)^{-(n_1+ n_2) \epsilon_1/2} \mathcal M_2f(x). 
\end{align}

Putting the bounds of \eqref{mainthm1} and \eqref{mainthm2} into \eqref{eq:M2-sparse-thm-estimate1-annulusdecomposition-Case1} we get 
\begin{align} \label{estimate2} 
I_{j} & \lesssim_{\epsilon_1} \left\{ \left( 2^{j/2} r_{0} \right)^{-(n_1+ n_2) \epsilon_1/2} + \left( 2^{j/2} r_{0} \right)^{-Q \epsilon_1/2} \right\} \mathcal M_2f(x) \sum_{l \geq 0} 2^{-l \epsilon_1 / 2} \\ 
\nonumber & \lesssim_{\epsilon_1} \left\{ \left( 2^{j/2} r_{0} \right)^{-(n_1+ n_2) \epsilon_1/2} + \left( 2^{j/2} r_{0} \right)^{-Q \epsilon_1/2} \right\} \mathcal M_2f(x).
\end{align}

\medskip \noindent \textbf{\underline{Case 1 ($r_0 \geq 1$)}:}

In this case, we apply the bounds of $I_{j}$ from \eqref{estimate2} into \eqref{eq:M2-sparse-thm-estimate1} to get 
\begin{align*}
& \left| T^{1} \left(f \chi_{\mathbb{R}^{n+n_2} \setminus B( z_0, r_0))} \right) (y) - T^{1} \left(f \chi_{\mathbb{R}^{n_1+n_2}\setminus B( z_0, r_0))} \right) (z) \right| \\ 
& \lesssim_{\epsilon_1} \sum_{j \geq 0} \left\{ \left( 2^{j/2} \right)^{-(n_1+ n_2) \epsilon_1/2} + \left( 2^{j/2} \right)^{-Q \epsilon_1/2} \right\} \mathcal M_2f(x) \lesssim_{\epsilon_1} \mathcal{M}_2f(x). 
\end{align*} 

\medskip \noindent \textbf{\underline{Case 2 ($r_0 < 1$)}:}

Let us choose and fix $j_{0}$ such that $r_{0} 2^{j_{0}/2} \sim 1$. Let us consider the infinite sum of \eqref{eq:M2-sparse-thm-estimate1}, and stack the pieces in the following way: 
\begin{align} \label{eq:M2-sparse-thm-estimate2} 
&|T^{1}(f \chi_{\mathbb{R}^{n+n_2}\setminus B( z_0, r_0))})(y)-T^{ 1}(f \chi_{\mathbb{R}^{n_1+n_2}\setminus B( z_0, r_0))})(z)|\leq \sum_{j=0}^{j_0} I_{j}+\sum_{j>j_{0}} I_{j}, 
\end{align}
where $I_j$'s are same as in \eqref{eq:M2-sparse-thm-estimate1}. 

For the infinite sum (that is, when $j > j_0$) in \eqref{eq:M2-sparse-thm-estimate2}, we use the bounds of $I_{j}$ from the estimate \eqref{estimate2}, and while doing so we also make use of the fact that $\left( 2^{j/2} r_0 \right)^{-(n_1+ n_2) \epsilon_1/2} + \left( 2^{j/2} r_0 \right)^{-Q \epsilon_1/2} \sim_{\epsilon_1} \left( 2^{j/2} r_0 \right)^{-(n_1+ n_2) \epsilon_1/2}$ uniformly for $j > j_0$. We can conclude that 
\begin{align} \label{estimate3}
\sum_{j > j_{0}} I_{j} & \lesssim_{\epsilon_1} r_0^{-(n_1+ n_2) \epsilon_1/2} \mathcal{M}_2 f(x) \sum_{j>j_{0}} 2^{- j (n_1+ n_2) \epsilon_1 / 4} \\ 
\nonumber & \lesssim_{\epsilon_1} r_0^{-(n_1+ n_2) \epsilon_1/2} \mathcal{M}_2f(x) 2^{-j_0(n_1+ n_2) \epsilon_1/4} \lesssim_{\epsilon_1} \mathcal M_2f(x),
\end{align}
where the last inequality follows from the relation $r_{0} 2^{j_{0}/2} \sim 1$. 

We are therefore left to estimate the finite sum (that is, when $j \leq j_0$) of \eqref{eq:M2-sparse-thm-estimate2}. In estimating these pieces, we make use of the gradient estimates of the kernels. Decompose each $I_{j}$ as in \eqref{eq:M2-sparse-thm-estimate1-annulusdecomposition-Case1} and note that
\begin{align}
|T^{1}_{j}(y,v) - T^{1}_{j}(z,v)| & = |T_{j}(y,v) \phi(d(y,v)) - T_{j}(z,v) \phi(d(z,v))| \label{GrushinMVT} \\
\nonumber & \leq |(T_{j}(y,v)-T_{j}(z,v))\phi(d(y,v))|+|T_{j}(z,v)(\phi(d(y,v))-\phi(d(z,v)))|. 
\end{align}

To estimate the first term in \eqref{GrushinMVT}, we make use of the mean-value estimate from Lemma \ref{lem:general-grushin-Mean-value}, to get that 
\begin{align} \label{Kernel-Mean-value}
|T_{j}(y,v) - T_{j}(z,v)| & \lesssim d(z, y) \int_0^1 \left| \left( X_x T_j \right) \left( \gamma_0(t), v \right) \right| dt, 
\end{align}
with $\gamma_0(t) \in B( z_0, C_{1, \varkappa} \, r)$. 

Therefore, we are left with estimating $d(z, y) \int_{\mathscr A_{l}} \left| \left( X_x T_j \right) \left(\gamma_0(t), v\right) \right| \, |f(v)| \, dv$ uniformly in $t\in (0, 1).$ Fix $t\in (0, 1).$ Now, as also done immediately after \eqref{eq:M2-sparse-thm-estimate1-annulusdecomposition-Case1}, we decompose the annulus $\mathscr A_{l}$ into two regions $ \left\{ v \in \mathscr A_{l} : |\gamma_0(t)'| \leq d( z_0,v) \right\}$ and $ \left\{ v \in \mathscr A_{l} : |\gamma_0(t)'| > d( z_0,v) \right\}.$ It is easy to see that with some obvious modification in the computations of \eqref{mainthm1}, one can obtain 
\begin{align} 
&\nonumber d(z, y) \int_{\{ v \in \mathscr A_{l} : |\gamma_0(t)'| \leq d( z_0,v)\}} \left| \left( X_x T_j \right) \left(\gamma_0(t), v\right) \right| \, |f(v)| \, dv \\ 
&\nonumber \lesssim d(y, z) \left(\frac{1}{|B(\gamma_0(t), r_{l})|^{\epsilon_1} |B(\gamma_0(t), r_{l+1})|} \int_{B( z_0, r_{l+1})}|f(v)|^2 \, dv \right)^{1/2} \\ 
\nonumber & \quad \times \left( \int_{\mathscr A_{l}} d( z_0,v)^{Q \left( 1 + \epsilon_1 \right)} \left| \left( X_x T_j \right) \left((\gamma_0(t), v \right) \right|^{2} \, \phi(d(y,v))^2 dv  \right)^{1/2} \\ 
\nonumber & \lesssim \frac{r_0}{\left( 2^l r_0^Q \right)^{\epsilon_1/2}} \left(\frac{1}{|B( z_0, r_{l+1})|} \int_{B( z_0, r_{l+1})}|f(v)|^2 \, dv \right)^{1/2} \\ 
& \quad \times \left(  \int_{\mathscr A_{l}} d(\gamma_0(t),v)^{Q(1+\epsilon_1)} \left| \left( X_x T_j \right) \left( (\gamma_0(t), v \right) \right|^{2} \phi(d(y,v))^2 \, dv  \right)^{1/2}. 
\label{imp-phi}
\end{align}
Since the function $\phi$ is supported on $[-5C^2_0,5C^2_0]$, we have $d(y,v)\leq  5C^2_0$ in the domain of the last integration. Also $d( z_0, v)$ and $d(y, v)$ are comparable since $r < 1$, $y\in B( z_0, r)$ and $v\in \mathbb{R}^{n_1+n_2} \setminus B( z_0, r_{0})$, hence $d( z_0, v)\leq_{\varkappa, C_{0}} d(y, v)\leq C_{\varkappa, C_{0}}$ in the range of the integration. Finally using the fact that $d( z_0, v)$ is comparable to $d(\gamma_0(t), v)$, condition \eqref{cond:General-hypo-grad} is applicable, and therefore the above term is dominated by 
\begin{align*} 
r_0 (2^l r_0^Q)^{-\epsilon_1/2} \, \mathcal{M}_2 f (x) \, 2^{-j Q (1 + \epsilon_1) /4} 2^{j/2} |B(\gamma_0 (t), 2^{-j/2})|^{-1/2} \lesssim \mathcal{M}_{2}f(x) \, 2^{- l \epsilon_1/2} \left( 2^{j/2} r_0 \right)^{1 - Q \epsilon_1 /2}.
\end{align*} 
Next, as in \eqref{mainthm2}, one can modify the above arguments to show that 
\begin{align*}
& d(z, y) \int_{\{ v \in \mathscr A_{l} : |\gamma_0(t)'| > d( z_0,v)\}} \left| \left( X_x T_j \right) \left(\gamma_0(t), v\right) \right| \, |f(v)| \, dv \\ 
& \lesssim_{\epsilon_1} \mathcal{M}_{2}f(x) \, 2^{- l \epsilon_1/2} \left( 2^{j/2} r_0 \right)^{1 - (n_1 + n_2) \epsilon_1 /2}.    
\end{align*}

Let us now consider the second term in \eqref{GrushinMVT}. By mean-value estimate from Lemma \ref{lem:general-grushin-Mean-value}, 
\begin{align*} 
|\phi(d(y,v)) - \phi(d(z,v))| & \lesssim d(z, y) \int_0^1 \left|   \phi^{\prime}  \left( d(\gamma_0(t), v) \right) \right| \left|Xd(\gamma_0(t), v) \right| dt. 
\end{align*}
From the definition of $\phi$ we see that $\phi^{\prime} \left( d(\gamma_0(t), v) \right)$ can be non-zero only when $3C^2_0 \leq d(\gamma_0(t), v) \leq 5 C_0^2$. Now, using the arguments of proof of part (iv) of Lemma $2.8$ of \cite{Bagchi-Garg-1}, in view of the lower and upper bound of the distance function where the derivative survives, we can conclude that $|Xd(\gamma_0(t), v)|$ is uniformly bounded for $\gamma_0(t) \in B( z_0, C_{1,\varkappa}r)$, $v \in \mathbb{R}^{n_1+n_2}\setminus B( z_0, r_0)$, and $r_0 < 1$. Therefore, the second term in \eqref{GrushinMVT} is dominated by $d(y,z) |T_j(z, v)|$. 

Decomposing the annulus $\mathscr{A}_{l}$ into two regions exactly as earlier, and in view of the presence of the extra term $d(y, z)$, we obtain
\begin{align*}
d(y, z) \int_{\mathscr{A}_{l}}|T_{j}(z, v)| \, |f(v)| \, dv & \lesssim_{\epsilon_1} r_0 2^{- l \epsilon_1 / 2} \left(2^{j/2} r_0 \right)^{-Q \epsilon_1 /2} \mathcal M_{2}f(x) \\ 
& \lesssim 2^{- l \epsilon_1 / 2} \left(2^{j/2} r_0 \right)^{1 - Q \epsilon_1 /2} \mathcal M_{2}f(x), 
\end{align*}
where we have used the fact that $2^{j_0/2} r_0 \sim 1$ implies 
$$\left( 2^{j/2} r_0 \right)^{-(n_1+ n_2) \epsilon_1/2} + \left( 2^{j/2} r_0 \right)^{-Q \epsilon_1/2} \sim_{\epsilon_1} \left( 2^{j/2} r_0 \right)^{-Q \epsilon_1/2}$$ 
uniformly for $j \leq j_0$. 

We can now estimate the finite sum $\sum_{j=0}^{j_0} I_{j}$ of \eqref{eq:M2-sparse-thm-estimate2} as follows: 
\begin{align} \label{mainthm5}
\sum_{j=0}^{j_{0}}I_{j} & \lesssim_{\epsilon_1} \mathcal{M}_{2}f(x)  \sum_{l=0}^{\infty} \sum_{j=0}^{j_0} 2^{- l \epsilon_1 / 2} \left(2^{j/2} r_0 \right)^{1 - Q \epsilon_1 /2} \\
\nonumber & \lesssim \mathcal{M}_{2}f(x) r_0^{1 - Q \epsilon_1 /2} \left( \sum_{j=0}^{j_0} 2^{j(1 - \frac{Q \epsilon_1}{2}) /2} \right) \left( \sum_{l \geq 0} 2^{- l \epsilon_1 / 2} \right) \\ 
\nonumber & \lesssim_{\epsilon_1} \mathcal{M}_{2}f(x) \left(2^{j_0/2} r_0 \right)^{1 - Q \epsilon_1 /2} \\ 
\nonumber & \lesssim_{\epsilon_1} \mathcal{M}_{2}f(x), 
\end{align}
where we have again used the fact $r_{0} 2^{j_{0}/2}\sim 1$ in the last inequality. 

\medskip \noindent \textbf{\underline{Estimation of $\mathfrak{I}_{2}$ in \eqref{eq:M2-sparse-thm-estimate0}}:} 

We follow the analysis which is very similar to the one for $\mathfrak{I}_{1}$. The key change here is that we do not need to make use of the mean-value estimate. Note that   
\begin{align}
\mathfrak{I}_{2}\leq \sum_{j\geq 0} \int_{\X\setminus B( z_0, r_{0})}|T^2_{j}(y, v)-T^2_{j}(z, v)| |f(v)| dv:=\sum_{j\geq 0} J_{j}.  
\end{align}
Let us write each $J_{j}$ as follows:
\begin{align}\label{T_2-kerenel-each piece} 
J_{j}\leq \int_{\X\setminus B( z_0, r_{0})}|T^2_{j}(y, v)||f(v)| \ dv + \int_{\X\setminus B( z_0, r_{0})}|T^2_{j}(z, v)||f(v)| \ dv. 
\end{align}

We only calculate the first term of right side of \eqref{T_2-kerenel-each piece}, as the calculation for second term is similar. 

\medskip \noindent \textbf{\underline{Case 1 ($r_0 \geq 1$)}:}
We have that 
\begin{align*}
\int_{\X\setminus B( z_0, r_{0})}|T^2_{j}(y, v)||f(v)| \ dv & = \int_{\X\setminus B( z_0, r_{0})}|T_{j}(y, v)|(1-\phi(d(y,v)) |f(v)| \ dv \\ 
& \leq  \int_{\X \setminus B( z_0, r_{0})}|T_{j}(y, v)||f(v)| \ dv.
\end{align*}
In this case, we can make estimation same as we did for $\mathfrak{I}_{1}$ to show that the above is dominated by 
$$ \left\{ \left( 2^{j/2} \right)^{-(n_1+ n_2) \epsilon_1/2} + \left( 2^{j/2} \right)^{-Q \epsilon_1/2} \right\} \mathcal M_2f(x), $$ 
and therefore if we take sum over $j\geq 0$ we get 
\begin{align*}
\sum_{j \geq 0} \left\{ \left( 2^{j/2} \right)^{-(n_1+ n_2) \epsilon_1/2} + \left( 2^{j/2} \right)^{-Q \epsilon_1/2} \right\} \mathcal M_2f(x) \lesssim_{\epsilon_1} \mathcal{M}_{2}f(x). 
\end{align*} 

\medskip \noindent \textbf{\underline{Case 2 ($r_0 < 1$)}:}
Since $\phi(t)=1$ for $|t|\leq 3C^2_0$, we have 
$$ \int_{\X\setminus B( z_0, r_{0})}|T^2_{j}(y, v)| |f(v)| \ dv \lesssim \int_{d(y, v)>3C^2_0}|T_{j}(y, v)| |f(v)| \ dv.$$

Since $r<1$ and $d(y,v)> 3C^2_0$, by triangle inequality we get $d(x,v)>C_0$. Therefore, the above term is dominated by 
\begin{align*}
\int_{d(x, v)>C_0}|T_{j}(y, v)||f(v)| \ dv.
\end{align*}

In the beginning of the proof of the theorem, we performed a decomposition of the space into annuli with keeping $ z_0$ as the centre. We can make an analogous decomposition with keeping $x$ as the centre. This is to ensure that the balls in the integral average contain $x$ and that would help us establishing bounds involving $\mathcal{M}_2 f(x)$. So, let us write $\mathfrak{s}_{0} = C_0$, and choose a sequence $\{\mathfrak{s}_{l}\}$ such that $\mathfrak{s}_{l+1}=2^{k_{l}}\mathfrak{s}_{l}$, $k_{l}\in\mathbb{N}$, with $|B(x,\mathfrak{s}_{l+1})|>2|B(x,\mathfrak{s}_{l})|$ and $|B(x,2^{k_{l}-1}\mathfrak{s}_{l})\leq 2|B(x,\mathfrak{s}_{l})|$. Let $\mathscr B_{l}$ denotes the annulas $B(x,\mathfrak{s}_{l+1})\setminus B(x,\mathfrak{s}_{l})$. Then, 
\begin{align*}
\int_{d(x,v)>C_0}|T_{j}(y,v)| |f(v)|\, dv 
&\leq \sum_{l}\int_{B(x,\mathfrak{s}_{l+1})\setminus B(x,\mathfrak{s}_{l})}|T_{j}(y,v)| |f(v)|\, dv \\
&=\sum_{l}\left(\int_{\mathscr B_{l}: |y^{\prime}|\leq  d(x,v)}+\int_{\mathscr{B}_{l}: |y^{\prime}|> d(x,v)}\right)|T_{j}(y,v)| |f(v)|\, dv.
\end{align*}
	
With the above terms, repeating the calculations as done earlier, it is easy to prove that 
\begin{align*}
\int_{\mathscr B_{l}: |y^{\prime}|\leq d(x,v)}|T_{j}(y,v)| |f(v)|\, dv & \lesssim_{\epsilon_1} 2^{- l \epsilon_1 / 2} \left(2^{j/2} r_0 \right)^{-Q \epsilon_1 /2} \mathcal M_{2}f(x), \\
\int_{\mathscr B_{l}:|y^{\prime}|> d(x,v)}|T_{j}(y,v)| |f(v)|\, dv & \lesssim_{\epsilon_1} 2^{-l \epsilon_1 / 2} \left( 2^{j/2} r_{0} \right)^{-(n_1+ n_2) \epsilon_1/2} \mathcal M_2f(x).
\end{align*}

Using the above inequalities and taking sum over $j \geq 0$ and $l \geq 0$, we get the desired estimate. This completes the proof of Theorem \ref{thm:domination-maximal-M2}. 
\end{proof}

\begin{remark} \label{rem:introduction-phi-function-minimal-smoothness}
We would like to highlight the importance of the decomposition of the operators $T_{j}$ into $T_{j}^1$ and $T_{j}^2$ in the proof of Theorem \ref{thm:domination-maximal-M2}. Since our goal is to prove Theorem~\ref{thm:pseudo-grushin-a=0-less-derivative} with $\lfloor Q/2\rfloor$ derivatives on the frequency variable ($\eta$-variable) of the function $X_{x}m(x, \eta),$ the presence of the function $\phi$ in the first term of \eqref{GrushinMVT} and subsequently in \eqref{imp-phi} allows us to reduce the integration on a compact set  and therefore we can apply condition \eqref{cond:General-hypo-grad}. Moreover, this fact combined with Corollary~\ref{cor-pseudo:grad-unweighted-L-2-estimate-compact} ensures that we only need $\lfloor Q/2\rfloor$ derivatives on the frequency variable ($\eta$-variable) of the function $X_{x}m(x, \eta).$ In principle, if one is not concerned with the optimality of the number of derivatives, such a decomposition can be avoided.
\end{remark}

%%%%%%%%%%%%%%%%%%%%%%%%%%%%%%%%%
%%%%%%%%%%%%%%%%%%%%%%%%%%%%%%%%%
%%%%%%%%%%%%%%%%%%%%%%%%%%%%%%%%%
%%%%%%%%%%%%%%%%%%%%%%%%%%%%%%%%%
%%%%%%%%%%%%%%%%%%%%%%%%%%%%%%%%%
%%%%%%%%%%%%%%%%%%%%%%%%%%%%%%%%%
%%%%%%%%%%%%%%%%%%%%%%%%%%%%%%%%%
%%%%%%%%%%%%%%%%%%%%%%%%%%%%%%%%%
%%%%%%%%%%%%%%%%%%%%%%%%%%%%%%%%%

\subsection{Proof of Theorem \ref{thm:main-sparse-more-derivative}} \label{subsec:proof-thm:main-sparse-more-derivative}
In this subsection, we develop the proof of Theorem \ref{thm:main-sparse-more-derivative}. In view of Theorem \ref{thm:Lorist-general-sparse-principle}, Theorem \ref{thm:main-sparse-more-derivative} follows from Theorems \ref{thm:domination-maximal-Mp} and \ref{thm:weak-type-bound-operator} which we shall establish.

We begin with Theorem \ref{thm:domination-maximal-Mp} which is about a sufficient condition ensuring the control of the grand truncated maximal operator $\mathcal{M}^{\sharp}_{T, s}$ by $\mathcal{M}$, which in turn would imply the weak type $(1, 1)$ boundedness for the grand truncated maximal operator $\mathcal{M}^{\sharp}_{T, s}$. 

\begin{theorem} \label{thm:domination-maximal-Mp}
Let $T \in \mathcal{B} \left( L^2({\mathbb{R}^{n_1+n_2}}) \right)$ be such that it admits a decomposition $T=\sum_{j\geq 0}T_{j}$, with the convergence in the strong operator topology, and that the integral kernels $T_j(x,y)$ satisfy conditions \eqref{cond:General-hypo-sup} and \eqref{cond:General-hypo-grad-sup} for some $R_0 > Q$. Then, for $s = 8 C_{1, \varkappa} C_{0}^3 \delta_0^{-1}$ with $\delta_0$ as in Theorem \ref{thm:dyadic-cubes} and $C_{1, \varkappa}$ as in Lemma \ref{lem:general-grushin-Mean-value}, we have the following pointwise almost everywhere estimate 
\begin{align}
\mathcal{M}^{\sharp}_{T, s}f(x) \lesssim_{T, \varkappa, s} \mathcal{M}f(x)
\end{align} 
for every bounded measurable function $f$ with compact support. 
\end{theorem}
\begin{proof}
The proof follows on exact same lines of the proof of Theorem \ref{thm:domination-maximal-M2}. The only difference is that in various integral estimations, we do not apply Cauchy-Schwarz inequality as we have already assumed pointwise weighted estimates of the kernels. For the sake of convenience, keeping notations of proof of Theorem \ref{thm:domination-maximal-M2}, let us repeat calculations of \eqref{mainthm1} in our case here. 
\begin{align} \label{mainthm1-M}
& \int_{\mathscr A_{l} \cap \{|y'| \leq d( z_0, v)\}}|T^{1}_{j}(y,v)| \, |f(v)|\, dv \\
\nonumber & \leq \int_{\mathscr A_{l} \cap \{|y'| \leq d( z_0, v)\}} |B(y, d( z_0, v))|^{-(1+\epsilon_1)} |f(v)| |B(y, d( z_0,v))|^{(1+\epsilon_1)} |T^{1}_{j}(y,v)|\, dv \\
\nonumber & \lesssim |B(y, r_{l})|^{-(1+\epsilon_1)} \int_{\mathscr A_{l} \cap \{|y'| \leq d( z_0, v)\}} |f(v)|  d( z_0, v)^{Q (1+\epsilon_1)} |T_{j}(y,v)| |\phi (y,v)| \, dv \\
\nonumber & \lesssim |B(y, r_{l})|^{-(1+\epsilon_1)} \int_{\mathscr A_{l} \cap \{|y'| \leq d( z_0, v)\}} |f(v)|  d(y, v)^{Q (1+\epsilon_1)} |T_{j}(y,v)| \, dv \\
\nonumber & \lesssim_{\epsilon_1} |B(y, r_{l})|^{-(1+\epsilon_1)} \int_{B( z_0, r_{l+1})} |f(v)| \frac{2^{-j Q (1+\epsilon_1) / 2}}{|B(y, 2^{-j/2})|^{1/2}|B(v, 2^{-j/2})|^{1/2}}  \, dv \\
\nonumber & \lesssim 2^{-j \epsilon_1 / 2} |B(y, r_{l})|^{-\epsilon_1} \mathcal{M}f(x) \\ 
\nonumber & \lesssim 2^{-l \epsilon_1} \left( r_0 2^{j/2} \right)^{-\epsilon_1} \, \mathcal M f(x). 
\end{align} 
Similar to the above one, we can derive all other estimates of proof of Theorem \ref{thm:domination-maximal-M2}, and we leave the details. 
\end{proof}

%%%%%%%%%%%%%%%%%%%%%%%%%%%%%%%%%%
%%%%%%%%%%%%%%%%%%%%%%%%%%%%%%%%%%
%%%%%%%%%%%%%%%%%%%%%%%%%%%%%%%%%%
%%%%%%%%%%%%%%%%%%%%%%%%%%%%%%%%%%

Finally, we discuss the following result concerning the weak type $(1,1)$-boundedness. 

\begin{theorem} \label{thm:weak-type-bound-operator} 
Let $T \in \mathcal{B} \left( L^2({\mathbb{R}^{n_1+n_2}}) \right)$ be such that it admits a decomposition $T=\sum_{j\geq 0}T_{j}$, with the convergence in the strong operator topology, and that the integral kernels $T_j(x,y)$ satisfy condition \eqref{cond:General-hypo-sup} and \eqref{cond:General-hypo-y-grad-sup} for some $R_0 \geq Q + \frac{1}{2}$, then $T$ is weak type $(1,1)$. 
\end{theorem}
\begin{proof}
We shall prove that the operators $\mathcal T_{N} = \sum_{j = 0}^N T_j$ are weak type $(1,1)$, with their operator norms $\| \mathcal T_{N} \|_{L^1 \to L^{1, \infty}}$ being uniform in $N \in \mathbb{N}$, which would imply that $T$ is also weak type $(1,1)$. To see this, let us assume that we have 
\begin{align} \label{ineq:weak-type-bound}
|\{x: |\mathcal{T}_Nf(x)|>\epsilon\}| \lesssim \frac{1}{\epsilon} \|f\|_{L^1}
\end{align}
for every $\epsilon> 0$, $N \in \mathbb{N}$ and $f \in \mathcal{S} (\mathbb{R}^{n_1 + n_2})$. 

Given $f \in \mathcal{S} (\mathbb{R}^{n_1 + n_2})$, since $\lim_{N \to \infty} \left\| \mathcal{T}_N f - T f \right\|_2 = 0$, it follows that $\mathcal{T}_N f$ converges to $T f$ in measure, that is, for every $\epsilon> 0$, 
\begin{align} \label{convergence-in-measure} 
\lim_{N \rightarrow \infty} |\{x: |\mathcal{T}_Nf(x)-Tf(x)|>\epsilon\}| =0.
\end{align}

In view of \eqref{ineq:weak-type-bound} and \eqref{convergence-in-measure}, 
\begin{align*}
|\{x: |Tf(x)|>\epsilon\}| & \leq |\{x: |\mathcal{T}_Nf(x)-Tf(x)|>\epsilon/2\}| + |\{x: |\mathcal{T}_Nf(x)|>\epsilon/2\}| \\ 
& \lesssim |\{x: |\mathcal{T}_Nf(x) - Tf(x)| > \epsilon/2 \}| + \frac{1}{\epsilon} \|f\|_{L^1} \\ 
& \xrightarrow{N \to \infty} \frac{1}{\epsilon} \|f\|_{L^1}, 
\end{align*}
establishing that $T$ is weak type $(1,1)$. 

We now claim that the following estimate holds true: 
\begin{align} \label{Hormander-piece}
d(x,y)^{1/2} |T_j(x,y)-T_j(x,z)| \lesssim \frac{d(y,z)^{1/2}}{|B(x,d(x,y))|} \min \left\{ \frac{2^{-j/4}}{d(y,z)^{1/2}},\frac{d(y,z)^{1/2}}{2^{-j/4}} \right\}, 
\end{align}
whenever $d(y, z) < \frac{1}{2 C_0} d(x, y)$. 

It follows from \eqref{Hormander-piece} that whenever $d(y, z) < \frac{1}{2 C_0} d(x, y)$ then we have 
\begin{align} \label{Hormander1} 
d(x,y)^{1/2} |\mathcal{T}_N(x,y) - \mathcal{T}_N(x,z)| \lesssim \frac{d(y,z)^{1/2}}{|B(x,d(x,y))|},
\end{align}
with the implicit bound uniform in $N$. 

With \eqref{Hormander1}, the proof of the weak type boundedness of $\mathcal{T}_N$ (together with the fact that the operator norms are uniform in $N$) follows from the uniform $L^2$-boundedness of $\mathcal{T}_N$ and the Calder\'on-Zygmund decomposition on homogeneous type spaces, for more details we refer to \cite{Coifman-Weiss-book-1971}. 

So, we proceed to establish \eqref{Hormander-piece}. 

\medskip \noindent \textbf{\underline{Case 1 (when $ |x'|\leq 2~ d(x,y)$)}:} 
In this case \eqref{Hormander-piece} is equivalent to 
\begin{align} \label{Hormander-piece-case1}
d(x,y)^{Q+ \frac{1}{2}} |T_j(x,y)-T_j(x,z)| \lesssim d(y,z)^{1/2}\min \left\{\frac{2^{-j/4}}{d(y,z)^{1/2}},\frac{d(y,z)^{1/2}}{2^{-j/4}}\right\}, 
\end{align}
whenever $d(y, z) < \frac{1}{2 C_0 C_{1,\varkappa}} d(x, y)$. 

Since $d(y, z) < \frac{1}{2 C_0 C_{1,\varkappa}} d(x, y)$, we have that $d(x, y)$ and $d(x, z)$ are comparable, and therefore we get from condition (\ref{cond:General-hypo-sup}) that 
\begin{align} \label{weak-ker} 
d(x,y)^{Q+\frac{1}{2}} |T_j(x,y)-T_j(x,z)| \lesssim d(y,z)^{1/2} \frac{2^{-j/4}}{d(y,z)^{1/2}}.
\end{align}

On the other hand, using mean-value estimate from Lemma \ref{lem:general-grushin-Mean-value}, we have 
\begin{align} \label{Grad-kernel-Mean-value}
d(x,y)^{Q+\frac{1}{2}}|T_j(x,y)-T_j(x,z)| & \lesssim  d(x,y)^{Q+\frac{1}{2}} d(y, z) \int_0^1 \left| \left( X_y T_j \right) \left(x, \gamma_0(t) \right) \right| dt.
\end{align}

Now from Lemma \ref{lem:general-grushin-Mean-value} we get $d(\gamma_0(t), y) \leq  C_{1,\varkappa} d(z,y)$. Thus, the assumption that $d(y, z) < \frac{1}{2 C_0 C_{1,\varkappa}} d(x, y)$ implies that $d(\gamma_0(t), y) < \frac{1}{2 C_0} d(x, y)$. And then, it follows that $d(x, y) \lesssim d(x, \gamma_0(t))$. With that, condition \eqref{cond:General-hypo-y-grad-sup} is applicable to imply that 
\begin{align}\label{Mean-value-y'}
d(x,y)^{Q+\frac{1}{2}} d(y, z) \int_0^1 \left| \left( X_y T_j \right) \left(x, \gamma_0(t)) \right) \right| dt \lesssim d(y, z) \, 2^{j/4} = d(y,z)^{1/2}\frac{d(y,z)^{1/2}}{2^{-j/4}}.
\end{align}

The above estimate implies that 
\begin{align}\label{weak-ker-grad}
d(x,y)^{Q+\frac{1}{2}}	|T_j(x,y)-T_j(x,z)| \lesssim d(y,z)^{1/2}\frac{d(y,z)^{1/2}}{2^{-j/4}}, 
\end{align}
and \eqref{weak-ker} and \eqref{weak-ker-grad} together establish the claimed estimate \eqref{Hormander-piece-case1}. 

\medskip \noindent \textbf{\underline{Case 2 (when $ |x'| > 2~ d(x,y)$)}:} 
In this case \eqref{Hormander-piece} is equivalent to 
\begin{align} \label{Hormander-piece-case2}
|x'|^{\varkappa n_2} d(x,y)^{n_1+n_2+\frac{1}{2}} |T_j(x,y)-T_j(x,z)| \lesssim d(y,z)^{1/2} \min \left\{ \frac{2^{-j/4}}{d(y,z)^{1/2}}, \frac{d(y,z)^{1/2}}{2^{-j/4}} \right\},
\end{align}
whenever $d(y, z) < \frac{1}{2 C_0 C_{1,\varkappa}} d(x, y)$. 

As earlier, from condition \eqref{cond:General-hypo-sup} and the fact that $d(x, y)$ and $d(x, z)$ are comparable whenever $d(y, z) < \frac{1}{2 C_0 C_{1,\varkappa}} d(x, y)$, we get 
\begin{align} \label{cent-ker}
|x'|^{\varkappa n_2}d(x,y)^{n_1+n_2+\frac{1}{2}} |T_j(x,y)-T_j(x,z)| \lesssim |x'|^{\varkappa n_2} \frac{2^{-j/4}}{|x'|^{\frac{\varkappa n_2}{2}} |y'|^{\frac{\varkappa n_2}{2}}} = 2^{-j/4}\frac{|x'|^\frac{\varkappa n_2}{2}}{|y'|^{\frac{\varkappa n_2}{2}}}.
\end{align}

Now, $|x'| > 2 d(x,y) \geq 2|x'-y'|\geq 2(|x'|-|y'|)$ implies that $|x'|\leq 2|y'|$. Therefore, we get from \eqref{cent-ker} that 
\begin{align}\label{cent-ker-1}
|x'|^{\varkappa n_2}d(x,y)^{n_1+n_2+\frac{1}{2}} |T_j(x,y)-T_j(x,z)| \lesssim 2^{-j/4} = d(y,z)^{1/2} \frac{2^{-j/4}}{d(y,z)^{1/2}}.
\end{align}

As in Case 1, one can apply the mean-value estimate to get the form analogous to  \eqref{Grad-kernel-Mean-value}, and then making use of condition \eqref{cond:General-hypo-y-grad-sup}  one can show that 
\begin{align} \label{cent-ker-grad}
|x'|^{\varkappa n_2}d(x,y)^{n_1+n_2+\frac{1}{2}} |T_j(x,y)-T_j(x,z)| & \lesssim d(y,z) \, |x'|^{\varkappa n_2} \frac{2^{j/4}}{|x'|^{\frac{\varkappa n_2}{2}} |y'|^{\frac{\varkappa n_2}{2}}} \\ 
\nonumber & \lesssim d(y,z)^{1/2} \frac{d(y,z)^{1/2}}{2^{-j/4}}. 
\end{align} 
where the last inequality follows from the fact that $|x'| \leq 2|y'|$. 

Combining \eqref{cent-ker-1} and \eqref{cent-ker-grad}, we have the claimed estimate \eqref{Hormander-piece-case2}, and this completes the proof of Theorem \ref{thm:weak-type-bound-operator}. 
\end{proof}

%%%%%%%%%%%%%%%%%%%%%%%%%%%%%%%%%%%%%%%%%%%%%
%%%%%%%%%%%%%%%%%%%%%%%%%%%%%%%%%%%%%%%%%%%%%
%%%%%%%%%%%%%%%%%%%%%%%%%%%%%%%%%%%%%%%%%%%%%
%%%%%%%%%%%%%%%%%%%%%%%%%%%%%%%%%%%%%%%%%%%%%
%%%%%%%%%%%%%%%%%%%%%%%%%%%%%%%%%%%%%%%%%%%%%
%%%%%%%%%%%%%%%%%%%%%%%%%%%%%%%%%%%%%%%%%%%%%
%%%%%%%%%%%%%%%%%%%%%%%%%%%%%%%%%%%%%%%%%%%%%
%%%%%%%%%%%%%%%%%%%%%%%%%%%%%%%%%%%%%%%%%%%%%
%%%%%%%%%%%%%%%%%%%%%%%%%%%%%%%%%%%%%%%%%%%%%
%%%%%%%%%%%%%%%%%%%%%%%%%%%%%%%%%%%%%%%%%%%%%
%%%%%%%%%%%%%%%%%%%%%%%%%%%%%%%%%%%%%%%%%%%%%
%%%%%%%%%%%%%%%%%%%%%%%%%%%%%%%%%%%%%%%%%%%%%
%%%%%%%%%%%%%%%%%%%%%%%%%%%%%%%%%%%%%%%%%%%%%
%%%%%%%%%%%%%%%%%%%%%%%%%%%%%%%%%%%%%%%%%%%%%

\section{Pseudo-multipliers associated to Grushin operators \texorpdfstring{$G_{\varkappa}$}{}} \label{sec:direct-pseudo-grushin}
In this section, we shall study pseudo-multiplier operators associated to Grushin operators $G_{\varkappa}$, and establish proof of Theorems \ref{thm:pseudo-grushin-a=0-unweighted}, \ref{thm:pseudo-grushin-a=0-less-derivative} and \ref{thm:pseudo-grushin-a=0-full-derivative}. We do so by analysing range of conditions on symbols that imply various conditions of the type listed between \eqref{cond:General-hypo} to \eqref{cond:General-hypo-grad-sup}. With that our claimed results would follow directly from Theorems \ref{thm:main-sparse},  \ref{thm:main-sparse-more-derivative} and \ref{thm:weak-type-bound-operator} in view of Theorem \ref{thm:Lorist-general-sparse-principle}. 

As seen in Section \ref{sec:kernel-estimates-and-pointwise-sparse-domination}, we would require weighted Plancherel estimates (pointwise as well as of $L^2$-type) for integral kernels $K_{m (x, G_{\varkappa})}(x, y)$ of pseudo-multiplier operators $m (x, G_{\varkappa})$. We shall show that the same could be established following (and slightly modifying) the ideas of \cite{AnhBuiDuongSpectralMultipliersBesovTriebelLizorkin,DuongOuhabazSikoraWeightedPlancherel2002JFA}. We establish these kernel estimates in the next subsection before we take up on our main results. 

%%%%%%%%%%%%%%%%%%%%%%%%%%%%%%%%%
%%%%%%%%%%%%%%%%%%%%%%%%%%%%%%%%%
%%%%%%%%%%%%%%%%%%%%%%%%%%%%%%%%%
%%%%%%%%%%%%%%%%%%%%%%%%%%%%%%%%%
%%%%%%%%%%%%%%%%%%%%%%%%%%%%%%%%%
%%%%%%%%%%%%%%%%%%%%%%%%%%%%%%%%%
%%%%%%%%%%%%%%%%%%%%%%%%%%%%%%%%%
%%%%%%%%%%%%%%%%%%%%%%%%%%%%%%%%%
%%%%%%%%%%%%%%%%%%%%%%%%%%%%%%%%%

\subsection{Weighted Plancherel (gradient) estimates} 
\label{subsec:gradient-weighted-Plancherel-estimates} 
Using estimates \eqref{est:derivative-bounds-heat-kernel-grushin}, a straight forward revision of the proof of Lemma 2.1 of \cite{DuongOuhabazSikoraWeightedPlancherel2002JFA} would lead to the following gradient estimates: 
\begin{align} \label{grad-heat-kernel}
|B(x, R^{-1})| \int_{\mathbb{R}^{n_1+n_2}} | X_x^{\Gamma} p_{R^{-2}}(x,y)|^{2} \, dy \lesssim_{\Gamma} R^{2|\Gamma|}, \\ 
\nonumber \text{and} \quad |B(x, R^{-1})| \int_{\mathbb{R}^{n_1+n_2}} | X_y^{\Gamma} p_{R^{-2}}(x,y)|^{2} \, dy \lesssim_{\Gamma} R^{2|\Gamma|}, 
\end{align} 
for all $\Gamma \in \mathbb{N}^{n_0}$.

Using estimates \eqref{grad-heat-kernel} and the $L^2$-boundedness of Riesz transforms of various orders (see \eqref{est:L2-boundedness-Riesz-transforms}), one can adapt the proof of Lemma 2.2 of \cite{DuongOuhabazSikoraWeightedPlancherel2002JFA} to establish the following unweighted Plancherel estimates for the gradients of integral kernel. 
\begin{lemma} \label{lem:grad-unweighted-L^2-estimate}
We have 
\begin{align} 
|B(x,R^{-1})| \int_{\mathbb{R}^{n_1+n_2}} \left| X_{x}^{\Gamma} K_{m (G_{\varkappa})}(x, y) \right|^2 dy & \lesssim_{\Gamma} R^{2|\Gamma|} \|m\|^2_{\infty},  \label{est:grad-unweighted-L^2-estimate-x-grad} \\ 
|B(x,R^{-1})| \int_{\mathbb{R}^{n_1+n_2}} \left| X_{y}^{\Gamma} K_{m (G_{\varkappa})}(x, y) \right|^2 dy & \lesssim_{\Gamma} R^{2|\Gamma|} \|m\|^2_{\infty} \label{est:grad-unweighted-L^2-estimate-y-grad}, 
\end{align} 
for all $\Gamma \in \mathbb{N}^{n_0}$ and for every bounded Borel function $m$ supported on $[0, R^2]$ for any $R>0$. 
\end{lemma} 
\begin{proof}
Let us write $F_1 (\eta) = m(\eta) e^{\eta / R^2}$ and $F_2 (\eta) = e^{-\eta / R^2}$. Then, $m(\eta) = F_1 (\eta) F_2 (\eta) = F_2 (\eta) F_1 (\eta)$, and therefore following the basic arguments of the proof of Lemma 2.2 of \cite{DuongOuhabazSikoraWeightedPlancherel2002JFA} we have 
\begin{align} \label{est:grad-unweighted-L^2-estimate-kernel-composition-F12} 
K_{m (G_{\varkappa})}(x, y) = \int_{\mathbb{R}^{n_1+n_2}} K_{F_1(G_{\varkappa})}(x, z) K_{F_2(G_{\varkappa})}(z, y) \, dz, 
\end{align} 
and similarly 
\begin{align} \label{est:grad-unweighted-L^2-estimate-kernel-composition-F21} 
K_{m (G_{\varkappa})}(x, y) = \int_{\mathbb{R}^{n_1+n_2}} K_{F_2(G_{\varkappa})}(x, z) K_{F_1(G_{\varkappa})}(z, y) \, dz. 
\end{align} 

\medskip \noindent \underline{\textbf{Step 1}: estimate \eqref{est:grad-unweighted-L^2-estimate-x-grad} is true.} 

Since $G_{\varkappa}$ is self-adjoint, it follows that the $K_{m (G_{\varkappa})} (x, y) = \overline{K_{\bar{m} (G_{\varkappa})} (y, x)} $, and therefore, proving \eqref{est:grad-unweighted-L^2-estimate-x-grad} is equivalent to showing that 
$$ \int_{\mathbb{R}^{n_1+n_2}} \left| X_{y}^{\Gamma} K_{\bar{m} (G_{\varkappa})}(x, y) \right|^2 dx \lesssim R^{2|\Gamma|} |B(y, R^{-1})|^{-1} \|m\|^2_{\infty},$$ 
which we now prove. 

In view of \eqref{est:grad-unweighted-L^2-estimate-kernel-composition-F12}, we have 
\begin{align*}
\int_{\mathbb{R}^{n_1+n_2}} \left| X_{y}^{\Gamma} K_{\bar{m} (G_{\varkappa})}(x, y) \right|^2 dx & \lesssim \left\| F_1(G_{\varkappa}) \right\|^2_{op} \left\| X_y^{\Gamma} K_{F_2(G_{\varkappa})} (\cdot, y) \right\|^2_2 \\ 
& = \left\| F_1(G_{\varkappa}) \right\|^2_{op} \left\| X_y^{\Gamma} p_{R^{-2}} (\cdot, y) \right\|^2_2 \\ 
& \lesssim_{\Gamma} \left\| F_1 \right\|^2_{\infty} R^{2 |\Gamma|} |B(y, R^{-1})|^{-1} \\ 
& \lesssim R^{2|\Gamma|} |B(y, R^{-1})|^{-1} \left\| m \right\|^2_{\infty},
\end{align*}
where the second last inequality follows from \eqref{grad-heat-kernel}. This completes the proof of the claimed estimate. 

\medskip \noindent \underline{\textbf{Step 2}: estimate \eqref{est:grad-unweighted-L^2-estimate-y-grad} is true.} 

As seen earlier, since $G_{\varkappa}$ is self-adjoint, it suffices to show that 
$$ \int_{\mathbb{R}^{n_1+n_2}} \left| X_{x}^{\Gamma} K_{\bar{m} (G_{\varkappa})}(x, y) \right|^2 dx \lesssim R^{2|\Gamma|} |B(y, R^{-1})|^{-1} \|m\|^2_ {L^{\infty}},$$ 
which we now prove. 

Let $Op \left( X_{x}^{\Gamma} K_{F_2(G_{\varkappa})} (x,y) \right)$ stand for the operator whose integral kernel is $X_{x}^{\Gamma} K_{F_2(G_{\varkappa})} (x,y)$. Then, in view of \eqref{est:grad-unweighted-L^2-estimate-kernel-composition-F21}, we have 
\begin{align} \label{est:grad-unweighted-L^2-estimate-grad-parts}
\int_{\mathbb{R}^{n_1+n_2}} \left| X_{x}^{\Gamma} K_{\bar{m} (G_{\varkappa})}(x, y) \right|^2 dx & \lesssim \left\| Op \left( X_{x}^{\Gamma} K_{F_2(G_{\varkappa})} (x,y) \right) \right\|^2_{op} \left\| K_{F_1(G_{\varkappa})} (\cdot, y) \right\|^2_2 \\ 
\nonumber & \lesssim_{\Gamma} \left\| Op \left( X_{x}^{\Gamma} K_{F_2(G_{\varkappa})} (x,y) \right) \right\|^2_{op} |B(y, R^{-1})|^{-1} \left\| F_1 \right\|^2_{\infty} \\ 
\nonumber & \lesssim_{\Gamma} \left\| Op \left( X_{x}^{\Gamma} K_{F_2(G_{\varkappa})} (x,y) \right) \right\|^2_{op} |B(y, R^{-1})|^{-1} \left\| m \right\|^2_{\infty},
\end{align}
where we have made use of Lemma 2.2 of \cite{DuongOuhabazSikoraWeightedPlancherel2002JFA} in the second last inequality. 

Now, to estimate $\left\| Op \left( X_{x}^{\Gamma} K_{F_2(G_{\varkappa})} (x,y) \right) \right\|_{op}$, take any $f, g \in \mathcal{S} \left( \mathbb{R}^{n_1+n_2} \right)$, then we have 
\begin{align} \label{est:grad-unweighted-L^2-estimate-cauchy-schwarz}
\left| \left( Op \left( X_{x}^{\Gamma} K_{F_2(G_{\varkappa})} (x,y) \right) f, \, g \right) \right| & = \left| \int_{\mathbb{R}^{n_1+n_2}} \int_{\mathbb{R}^{n_1+n_2}} X_{x}^{\Gamma} K_{F_2(G_{\varkappa})} (x,y) f(y) \, \overline{g(x)} \, dy \, dx \right| \\ 
\nonumber & = \left| \int_{\mathbb{R}^{n_1+n_2}} \int_{\mathbb{R}^{n_1+n_2}} K_{F_2(G_{\varkappa})} (x,y) f(y) \, \overline{X_{x}^{\Gamma} g(x)} \, dy \, dx  \right| \\ 
\nonumber & = \left| \left( e^{- R^{-2} G_{\varkappa}} f, \, X_{x}^{\Gamma} g \right) \right| \\ 
\nonumber & = \left| \left( G_{\varkappa}^{|\Gamma|/2} e^{- R^{-2} G_{\varkappa}} f, \, G_{\varkappa}^{-|\Gamma|/2} X_{x}^{\Gamma} g \right) \right| \\ 
\nonumber & \leq \left\| G_{\varkappa}^{|\Gamma|/2} e^{- R^{-2} G_{\varkappa}} f \right\|_2 \left\| G_{\varkappa}^{-|\Gamma|/2} X_{x}^{\Gamma} g \right\|_2 \\ 
\nonumber & \lesssim_{\Gamma} \left\| G_{\varkappa}^{|\Gamma|/2} e^{- R^{-2} G_{\varkappa}} f \right\|_2 \, \left\| g \right\|_2, 
\end{align}
where we have used \eqref{est:L2-boundedness-Riesz-transforms} in the last inequality. 

Finally, note from \eqref{Gru-pseudo} that $G_{\varkappa}^{|\Gamma|/2} e^{- R^{-2} G_{\varkappa}} $ is a Grushin multiplier operator with symbol function $F_{R, \Gamma}$, where $F_{R, \Gamma} (\eta) = \eta^{|\Gamma|/2} e^{-R^{-2} \eta}$. Therefore, 
$$\left\| G_{\varkappa}^{|\Gamma|/2} e^{- R^{-2} G_{\varkappa}} f \right\|_2 \leq \| F_{R, \Gamma} \|_{\infty} \left\| f \right\|_2 \lesssim_{\Gamma} R^{|\Gamma|} \left\| f \right\|_2,$$
and putting this estimate in \eqref{est:grad-unweighted-L^2-estimate-cauchy-schwarz} we get 
$$\left\| Op \left( X_{x}^{\Gamma} K_{F_2(G_{\varkappa})} (x,y) \right) \right\|_{op} \lesssim_{\Gamma} R^{|\Gamma|}.$$
The claim of Step 2 now follows by using the above estimate in \eqref{est:grad-unweighted-L^2-estimate-grad-parts}. 

This completes the proof of Lemma \ref{lem:grad-unweighted-L^2-estimate}. 
\end{proof}

%%%%%%%%%%%%%%%%%%%%%%%%%%
%%%%%%%%%%%%%%%%%%%%%%%%%%
%%%%%%%%%%%%%%%%%%%%%%%%%%
%%%%%%%%%%%%%%%%%%%%%%%%%%
%%%%%%%%%%%%%%%%%%%%%%%%%%
%%%%%%%%%%%%%%%%%%%%%%%%%%
%%%%%%%%%%%%%%%%%%%%%%%%%%

One can deduce the following analogue of Lemma \ref{lem:grad-unweighted-L^2-estimate} for pseudo-multipliers. 
\begin{corollary} \label{cor-pseudo:grad-unweighted-L^2-estimate}
We have 
\begin{align} 
|B(x,R^{-1})| \int_{\mathbb{R}^{n_1+n_2}} \left| X_{x}^{\Gamma} K_{m (x, G_{\varkappa})}(x, y) \right|^2 dy & \lesssim_{\Gamma} \sup_{x_{0} \in \mathbb{R}^{n_1+n_2}} \sum_{\Gamma_1 + \Gamma_2 = \Gamma} R^{2|\Gamma_1|} \|X^{\Gamma_2}_{x} m(x_0, \cdot) \|^2_{\infty}, \label{est-pseudo:grad-unweighted-L^2-estimate-x-grad} \\ 
|B(x,R^{-1})| \int_{\mathbb{R}^{n_1+n_2}} \left| X_{y}^{\Gamma} K_{m (x, G_{\varkappa})}(x, y) \right|^2 dy & \lesssim_{\Gamma} \sup_{x_{0} \in \mathbb{R}^{n_1+n_2}} R^{2|\Gamma|} \| m(x_0, \cdot) \|^2_{\infty} \label{est-pseudo:grad-unweighted-L^2-estimate-y-grad}, 
\end{align} 
for all $\Gamma \in \mathbb{N}^{n_0}$ and for every bounded Borel function $m : \mathbb{R}^{n_1 + n_2} \times \mathbb{R} \rightarrow \mathbb{C}$ whose support in the last variable is in $[0, R^2]$ for any $R>0$. 
\end{corollary} 
\begin{proof} We write below a proof of \eqref{est-pseudo:grad-unweighted-L^2-estimate-x-grad} as a straightforward consequence of Lemma \ref{lem:grad-unweighted-L^2-estimate} by an application of the Leibniz formula. Note from \eqref{Gru-pseudo} that  
\begin{align*}
K_{m (x, G_{\varkappa})}(x, y) = \int_{\mathbb{R}^{n_2}} e^{-i \lambda \cdot \left( x^{\prime \prime} - y^{\prime \prime} \right)} \sum_{k \in \mathbb{N}} m \left( x, |\lambda|^{\frac{2}{\varkappa + 1}} \nu_{\varkappa, k} \right)  h^{\lambda}_{\varkappa, k} (x') \, h^{\lambda}_{\varkappa, k} (y')\, d\lambda, 
\end{align*}
where the integral and the infinite sum is absolutely convergent. This is because 
\begin{align*}
& \left( \int_{\mathbb{R}^{n_2}} \sum_{k \in \mathbb{N}} \left| m \left( x, |\lambda|^{\frac{2}{\varkappa + 1}} \nu_{\varkappa, k} \right)  h^{\lambda}_{\varkappa, k} (x') \, h^{\lambda}_{\varkappa, k} (y') \right| \, d\lambda \right)^2 \\ 
& \quad \leq \sup_{x_{0} \in \mathbb{R}^{n_1 + n_2}} \int_{\mathbb{R}^{n_2}} \sum_{k \in \mathbb{N}} \left| m \left( x_0, |\lambda|^{\frac{2}{\varkappa + 1}} \nu_{\varkappa, k} \right) \right| \left| h^{\lambda}_{\varkappa, k} (x') \right|^2 \, d\lambda \\ 
& \quad \quad \times \int_{\mathbb{R}^{n_2}} \sum_{k \in \mathbb{N}} \left| m \left( x_0, |\lambda|^{\frac{2}{\varkappa + 1}} \nu_{\varkappa, k} \right) \right| \left| h^{\lambda}_{\varkappa, k} (y') \right|^2 \, d\lambda \\ 
& \quad = \sup_{x_{0} \in \mathbb{R}^{n_1 + n_2}} \left( \int_{\mathbb{R}^{n_1+n_2}} \left| K_{\sqrt{|m|} (x_0, G_{\varkappa})}(x, z) \right|^2 \, dz \right) \left( \int_{\mathbb{R}^{n_1+n_2}} \left| K_{\sqrt{|m|} (x_0, G_{\varkappa})}(y, z) \right|^2 \, dz \right), 
\end{align*}
where the last equality follows from the orthogonality of the functions $h^{\lambda}_{\varkappa, k}$. 

The above process also suggests that one can make use of Fubini's theorem to justify taking the gradient in $x$ or $y$-variable inside the integral, and we have 
\begin{align*}
& \left| X_{x}^{\Gamma} K_{m (x, G_{\varkappa})}(x, y) \right| \\ 
& = \left| \int_{\mathbb{R}^{n_2}} e^{i \lambda \cdot y^{\prime \prime}} \sum_{k \in \mathbb{N}} X_{x}^{\Gamma} \left\{ m \left( x, |\lambda|^{\frac{2}{\varkappa + 1}} \nu_{\varkappa, k} \right)  e^{-i \lambda \cdot x^{\prime \prime}} h^{\lambda}_{\varkappa, k} (x') \right\} h^{\lambda}_{\varkappa, k} (y')\, d\lambda \right| \\ 
& = \left| \sum_{\Gamma_1 + \Gamma_2 = \Gamma} \binom{\Gamma}{\Gamma_1} \int_{\mathbb{R}^{n_2}} e^{i \lambda \cdot y^{\prime \prime}} \sum_{k \in \mathbb{N}} X_{x}^{\Gamma_2} \left\{ m \left( x, |\lambda|^{\frac{2}{\varkappa + 1}} \nu_{\varkappa, k} \right) \right\} X_{x}^{\Gamma_1} \left\{ e^{-i \lambda \cdot x^{\prime \prime}} h^{\lambda}_{\varkappa, k} (x') \right\} h^{\lambda}_{\varkappa, k} (y') \, d\lambda \right| \\ 
& \lesssim_{\Gamma} \sup_{x_{0}} \sum_{\Gamma_1 + \Gamma_2 = \Gamma} \left| \int_{\mathbb{R}^{n_2}} e^{i \lambda \cdot y^{\prime \prime}} \sum_{k \in \mathbb{N}} X_{x}^{\Gamma_2} \left\{ m \left( x_0, |\lambda|^{\frac{2}{\varkappa + 1}} \nu_{\varkappa, k} \right) \right\} X_{x}^{\Gamma_1} \left\{ e^{-i \lambda \cdot x^{\prime \prime}} h^{\lambda}_{\varkappa, k} (x') \right\} h^{\lambda}_{\varkappa, k} (y') \, d\lambda \right| \\ 
& = \sup_{x_{0}} \sum_{\Gamma_1 + \Gamma_2 = \Gamma} \left| X_{x}^{\Gamma_1} K_{X_{x}^{\Gamma_2} m (x_0, G_{\varkappa})}(x, y) \right|. 
\end{align*}

Now, one can make use of Lemma \ref{lem:grad-unweighted-L^2-estimate} to get 
\begin{align*}
\int_{\mathbb{R}^{n_1+n_2}} \left| X_{x}^{\Gamma} K_{m (x, G_{\varkappa})}(x, y) \right|^2 dy & \lesssim_{\Gamma} \sup_{x_{0} \in \mathbb{R}^{n_1+n_2}} \sum_{\Gamma_1 + \Gamma_2 = \Gamma} \int_{\mathbb{R}^{n_1+n_2}} \left| X_{x}^{\Gamma_1} K_{X_{x}^{\Gamma_2} m (x_0, G_{\varkappa})}(x, y) \right|^2 dy \\ 
& \lesssim_{\Gamma} |B(x,R^{-1})|^{-1} \sup_{x_{0} \in \mathbb{R}^{n_1+n_2}} \sum_{\Gamma_1 + \Gamma_2 = \Gamma} R^{2|\Gamma_1|} \|X^{\Gamma_2}_{x} m(x_0, \cdot) \|^2_{\infty}, 
\end{align*}
which completes the proof of \eqref{est-pseudo:grad-unweighted-L^2-estimate-x-grad}. 

The proof of \eqref{est-pseudo:grad-unweighted-L^2-estimate-y-grad} is similar, with the observation that the $y$-gradient does not fall on the symbol function $m$. This completes the proof of Corollary \ref{cor-pseudo:grad-unweighted-L^2-estimate}. 
\end{proof}

%%%%%%%%%%%%%%%%%%%%%%%%%%
%%%%%%%%%%%%%%%%%%%%%%%%%%
%%%%%%%%%%%%%%%%%%%%%%%%%%
%%%%%%%%%%%%%%%%%%%%%%%%%%
%%%%%%%%%%%%%%%%%%%%%%%%%%
%%%%%%%%%%%%%%%%%%%%%%%%%%
%%%%%%%%%%%%%%%%%%%%%%%%%%

Having established Corollary \ref{cor-pseudo:grad-unweighted-L^2-estimate} for the unweighted Plancherel estimates with gradients of integral kernels, we shall now prove the weighted Plancherel estimates with gradients of integral kernels. For the same, note first that using \eqref{grad-heat-kernel}, one can essentially repeat the proof of Lemma 4.1 of \cite{DuongOuhabazSikoraWeightedPlancherel2002JFA} to show that for any $\mathfrak{r} > 0$, 
\begin{align} 
|B(x, R^{-1})| \int_{\mathbb{R}^{n_1+n_2}} d(x,y)^{\mathfrak{r}} |X_x^{\Gamma} p_{(1 + i \tilde{\eta}) R^{-2}}(x,y)|^2 \, dy & \lesssim_{\Gamma, \mathfrak{r}} R^{2 |{\Gamma}|} R^{-\mathfrak{r}} (1 + |\tilde{\eta}|)^\mathfrak{r}, \label{x-grad-heat-kernel-complexified} \\ 
\nonumber \text{and} \quad |B(x, R^{-1})| \int_{\mathbb{R}^{n_1+n_2}} d(x,y)^{\mathfrak{r}} |X_y^{\Gamma} p_{(1 + i \tilde{\eta}) R^{-2}}(x,y)|^2 \, dy & \lesssim_{\Gamma, \mathfrak{r}} R^{2 |{\Gamma}|} R^{-\mathfrak{r}} (1 + |\tilde{\eta}|)^\mathfrak{r}, 
\end{align} 
where $p_{(1 + i \tilde{\eta})R^{-2}}$ is the integral kernel of the operator $ \exp \left( -(1 + i \tilde{\eta})R^{-2} G_{\varkappa} \right)$ with $\tilde{\eta} \in \mathbb{R}$. 

Then, making use of \eqref{est:derivative-bounds-heat-kernel-grushin} and Lemma \ref{lem:grad-unweighted-L^2-estimate}, we work with the ideas of the proof of Lemma 4.3 of \cite{DuongOuhabazSikoraWeightedPlancherel2002JFA} to prove the following result. 
\begin{lemma} \label{lem-pseudo:grad-weighted-L^2-estimate}
For every $\mathfrak{r}, \epsilon > 0$, we have 
\begin{align} 
& |B(x,R^{-1})| \int_{\mathbb{R}^{n_1+n_2}} (1+Rd(x,y))^{2\mathfrak{r}} \left| X_{x}^{\Gamma} K_{m (x, G_{\varkappa})}(x, y) \right|^2  dy \label{est:grad-weighted-L^2-estimate-x-grad} \\
\nonumber & \quad \lesssim_{\Gamma, \mathfrak{r}, \epsilon} \sup_{x_{0} \in \mathbb{R}^{n_1+n_2}} \sum_{\Gamma_1 + \Gamma_2 = \Gamma} R^{2|\Gamma_1|} \|X^{\Gamma_2}_{x} m(x_0, R^2 \cdot) \|^2_{W^{\infty}_{\mathfrak{r} + \epsilon}}, \\ 
& |B(x,R^{-1})| \int_{\mathbb{R}^{n_1+n_2}} (1+Rd(x,y))^{2\mathfrak{r}} \left| X_{y}^{\Gamma} K_{m (x, G_{\varkappa})}(x, y) \right|^2  dy \label{est:grad-weighted-L^2-estimate-y-grad} \\
\nonumber & \quad \lesssim_{\Gamma, \mathfrak{r}, \epsilon} \sup_{x_{0} \in \mathbb{R}^{n_1+n_2}} R^{2|\Gamma|} \|m(x_0, R^2 \cdot) \|^2_{W^{\infty}_{\mathfrak{r} + \epsilon}}, 
\end{align} 
for all $\Gamma \in \mathbb{N}^{n_0}$ and for every bounded Borel function $m : \mathbb{R}^{n_1 + n_2} \times \mathbb{R} \rightarrow \mathbb{C}$ whose support in the last variable is in $[0, R^2]$ for any $R>0$. 
\end{lemma}
\begin{proof}
We shall sketch the proof of Lemma \ref{lem-pseudo:grad-weighted-L^2-estimate} only mentioning the changes required in the proof of Lemma 4.3(a) of \cite{DuongOuhabazSikoraWeightedPlancherel2002JFA}. Let us define $F(x, \eta) = m(x, R^2 \eta) e^{\eta}$. Then, by Fourier inversion formula we have 
\begin{align*} 
F(x, G_{\varkappa}/R^2) e^{-G_{\varkappa}/R^2} = \frac{1}{2 \pi} \int_{\mathbb{R}} \exp \left( (i \tilde{\eta} - 1) R^{-2} G_{\varkappa} \right) \, \widehat{F}(x, \tilde{\eta}) \, d \tilde{\eta}, 
\end{align*}
where $\widehat{F}(x, \tilde{\eta})$ stands for the Fourier transform of $F$ as a function of $\eta$-variable (that is, with $x$-fixed), and as a consequence of the above relation we have 
\begin{align*}
K_{m(x, G_{\varkappa})} (x,y) = \frac{1}{2\pi}\int_{\mathbb{R}} \widehat{F}(x, \tilde{\eta}) p_{(1 - i \tilde{\eta}) R^{-2}}(x,y) \, d \tilde{\eta}, 
\end{align*}
which implies that 
\begin{align} \label{eq:grad-weighted-L^2-estimate-step1}
X_x^{\Gamma} K_{m(x, G_{\varkappa})} (x,y) & = \frac{1}{2\pi} \sum_{\Gamma_1 + \Gamma_2 = \Gamma} \binom{\Gamma}{\Gamma_1} \int_{\mathbb{R}} \left\{ \widehat{X_x^{\Gamma_2} F}(x, \tilde{\eta}) \right\} \left\{ X_x^{\Gamma_1} p_{(1 - i \tilde{\eta}) R^{-2}}(x,y) \right\} d \tilde{\eta}, 
\end{align}
and therefore 
\begin{align} \label{est:grad-weighted-L^2-estimate-step2} 
& \int_{\mathbb{R}^{n_1+n_2}} (1+Rd(x,y))^{2\mathfrak{r}} \left| X_{x} K_{m (x, G_{\varkappa})}(x, y) \right|^2 dy \\
\nonumber & \lesssim_{\Gamma} \sum_{\Gamma_1 + \Gamma_2 = \Gamma} \int_{\mathbb{R}^{n_1+n_2}} (1+Rd(x,y))^{2\mathfrak{r}} \left| \int_{\mathbb{R}} \left\{ \widehat{X_x^{\Gamma_2} F}(x, \tilde{\eta}) \right\} \left\{ X_x^{\Gamma_1} p_{(1 - i \tilde{\eta}) R^{-2}}(x,y) \right\} d \tilde{\eta} \right|^2 dy \\ 
\nonumber & \lesssim_{\Gamma} \sup_{x_0} \sum_{\Gamma_1 + \Gamma_2 = \Gamma} \int_{\mathbb{R}^{n_1+n_2}} (1+Rd(x,y))^{2\mathfrak{r}} \left| \int_{\mathbb{R}} \left\{ \widehat{X_x^{\Gamma_2} F}(x_0, \tilde{\eta}) \right\} \left\{ X_x^{\Gamma_1} p_{(1 - i \tilde{\eta}) R^{-2}}(x,y) \right\} d \tilde{\eta} \right|^2 dy 
\end{align}	

For each fixed $x_0 \in \mathbb{R}^{n_1 + n_2}$, each term in the final sum of \eqref{est:grad-weighted-L^2-estimate-step2} corresponds to Grushin multiplier symbol $X_x^{\Gamma_2} F(x_0, \eta)$, so the proof of Lemma 4.3(a) of \cite{DuongOuhabazSikoraWeightedPlancherel2002JFA} is applicable with the only change pertaining to the fact that the analogous estimation should be done with respect to $X_x^{\Gamma_1} p_{(1 - i \tilde{\eta}) R^{-2}}(x,y)$. Overall, one can repeat the proof of Lemma 4.3(a) of \cite{DuongOuhabazSikoraWeightedPlancherel2002JFA} (leading to estimates (4.4) and (4.5) in \cite{DuongOuhabazSikoraWeightedPlancherel2002JFA}) with the help of \eqref{grad-heat-kernel} and \eqref{x-grad-heat-kernel-complexified} to get that 
\begin{align} \label{est:grad-weighted-L^2-estimate-step4} 
& \int_{\mathbb{R}^{n_1+n_2}} (1+Rd(x,y))^{2\mathfrak{r}} \left| \int_{\mathbb{R}} \left\{ \widehat{X_x^{\Gamma_2} F}(x_0, \tilde{\eta}) \right\} \left\{ X_x^{\Gamma_1} p_{(1 - i \tilde{\eta}) R^{-2}}(x,y) \right\} d \tilde{\eta} \right|^2 dy \\ 
\nonumber & \quad \lesssim_{\Gamma, \mathfrak{r}, \epsilon} |B(x,R^{-1})|^{-1} \sup_{x_{0} \in \mathbb{R}^{n_1+n_2}} R^{2|\Gamma_1|} \|X^{\Gamma_2}_{x} m(x_0, R^2 \cdot) \|^2_ {W^{\infty}_{\mathfrak{r} + \frac{1}{2} + \epsilon}}. 
\end{align}	

Putting \eqref{est:grad-weighted-L^2-estimate-step4} into \eqref{est:grad-weighted-L^2-estimate-step2}, we get the claimed estimate \eqref{est:grad-weighted-L^2-estimate-x-grad} modulo a loss of an extra $1/2$ order of differentiability in the Sobolev norm. As explained in \cite{DuongOuhabazSikoraWeightedPlancherel2002JFA}, this loss could be removed by an interpolation trick utilising the unweighted Plancherel estimates. We have these  unweighted Plancherel estimates in Corollary \ref{cor-pseudo:grad-unweighted-L^2-estimate}, so we can repeat the arguments of \cite{DuongOuhabazSikoraWeightedPlancherel2002JFA} in our case too. That would establish estimate \eqref{est:grad-weighted-L^2-estimate-x-grad}. The proof of \eqref{est:grad-weighted-L^2-estimate-y-grad} could be written in the exact same way. 
This completes the proof of Lemma \ref{lem-pseudo:grad-weighted-L^2-estimate}. 
\end{proof}

%%%%%%%%%%%%%%%%%%%%%%%%%%%%%
%%%%%%%%%%%%%%%%%%%%%%%%%%%%%
%%%%%%%%%%%%%%%%%%%%%%%%%%%%%
%%%%%%%%%%%%%%%%%%%%%%%%%%%%%

Next, one can revise the proofs of Lemma 4.2 and 4.3 of \cite{AnhBuiDuongSpectralMultipliersBesovTriebelLizorkin} in a similar manner as we did above, and we can have pointwise unweighted and weighted Plancherel estimates for gradients. Once we have the pointwise estimates, we can do interpolation with estimates from Lemma \ref{lem-pseudo:grad-weighted-L^2-estimate} and Corollary \ref{cor-pseudo:grad-unweighted-L^2-estimate}, and thus we get the following estimates. 

\begin{corollary} \label{cor-pseudo:grad-unweighted-L-p-estimate}
For $2 \leq p \leq \infty$ we have 
\begin{align*} 
|B(x,R^{-1})|^{1/2} \left\| |B(\cdot, R^{-1})|^{1/2 - 1/p} X_{x}^{\Gamma} K_{m(x, G_{\varkappa})} (x, \cdot) \right\|_p & \lesssim_{\Gamma, p, \epsilon} \sup_{x_{0}} \sum_{\Gamma_1 + \Gamma_2 = \Gamma} R^{|\Gamma_1|} \|X^{\Gamma_2}_{x} m(x_0, \cdot) \|_{\infty}, \\ 
|B(x,R^{-1})|^{1/2} \left\| |B(\cdot, R^{-1})|^{1/2 - 1/p} X_{y}^{\Gamma} K_{m(x,  G_{\varkappa})} (x,\cdot) \right\|_p & \lesssim_{\Gamma, p, \epsilon} \sup_{x_{0}} R^{|\Gamma|} \| m(x_0, \cdot) \|_{\infty}, 
\end{align*} 
for every bounded Borel function $m : \mathbb{R}^{n_1 + n_2} \times \mathbb{R} \rightarrow \mathbb{C}$ whose support in the last variable is in $[0, R^2]$ for any $R>0$. 
\end{corollary} 

\begin{corollary} \label{cor-pseudo:grad-weighted-L-p-estimate}
For $2 \leq p \leq \infty$ and every $\mathfrak{r} > 0$ and $\epsilon > 0$, we have 
\begin{align*} 
& |B(x,R^{-1})|^{1/2} \left\| |B(\cdot, R^{-1})|^{1/2 - 1/p} \, (1+Rd(x,\cdot))^{\mathfrak{r}} X_{x}^{\Gamma} K_{m(x, G_{\varkappa})} (x,\cdot) \right\|_p \\ 
& \quad \lesssim_{\Gamma, p, \mathfrak{r} , \epsilon} \sup_{x_{0}} \sum_{\Gamma_1 + \Gamma_2 = \Gamma} R^{|\Gamma_1|} \|X^{\Gamma_2}_{x} m(x_0,R^2 \cdot) \|_{W_{\mathfrak{r} + \epsilon}^{\infty}}, \\ 
& |B(x,R^{-1})|^{1/2} \left\| |B(\cdot, R^{-1})|^{1/2 - 1/p} \, (1+Rd(x,\cdot))^{\mathfrak{r}} X_{y}^{\Gamma} K_{m(x, G_{\varkappa})} (x,\cdot) \right\|_p \\ 
& \quad \lesssim_{\Gamma, p, \mathfrak{r} , \epsilon} \sup_{x_{0}} R^{|\Gamma|} \| m(x_0, R^2\cdot) \|_{W_{\mathfrak{r} + \epsilon}^{\infty}}, 
\end{align*} 
for all $\Gamma \in \mathbb{N}^{n_0}$ and for every bounded Borel function $m : \mathbb{R}^{n_1 + n_2} \times \mathbb{R} \rightarrow \mathbb{C}$ whose support in the last variable is in $[0, R^2]$ for any $R>0$. 
\end{corollary} 

%%%%%%%%%%%%%%%%%%%%%%%%%%%%%
%%%%%%%%%%%%%%%%%%%%%%%%%%%%%
%%%%%%%%%%%%%%%%%%%%%%%%%%%%%
%%%%%%%%%%%%%%%%%%%%%%%%%%%%%
%%%%%%%%%%%%%%%%%%%%%%%%%%%%%
%%%%%%%%%%%%%%%%%%%%%%%%%%%%%
%%%%%%%%%%%%%%%%%%%%%%%%%%%%%

Before moving further, we record here weighted Plancherel estimates in a special case where we take the $L^p$-norm in $y$-variable over a compact set. This would help us in having a control over the total number of derivatives (space and spectral derivatives together) of the involved symbol function $m(x, \eta)$. 

\begin{corollary} \label{cor-pseudo:grad-unweighted-L-2-estimate-compact}
For every $\mathfrak{r} \geq 1$, $\epsilon > 0$, $0\leq \delta < 1$, and positive real number $\mathcal{K}_0$, we have 
\begin{align} \label{est:grad-weighted-L^2-estimate-x-grad-compact} 
& |B(x,R^{-1})| \int_{d(x,y)<\mathcal{K}_0} d(x,y)^{2\mathfrak{r}} \left| X_{x} K_{m (x, G_{\varkappa})}(x, y) \right|^2  dy\\
\nonumber & \lesssim_{\mathfrak{r}, \epsilon, \delta} \sup_{x_{0}} \left( R^{-2 (\mathfrak{r} - (1-\delta))} \|X_{x} m(x_0, R^2 \cdot) \|^2_{W^{\infty}_{\mathfrak{r}- (1 - \delta) + \epsilon}} + R^{-2 (\mathfrak{r} - 1)} \| m(x_0, R^2 \cdot) \|^2_{W^{\infty}_{\mathfrak{r}+ \epsilon}} \right), 
\end{align}
for every bounded Borel function $m : \mathbb{R}^{n_1 + n_2} \times \mathbb{R} \rightarrow \mathbb{C}$ whose support in the last variable is in $[0, R^2]$ for any $R>0$. 
\end{corollary} 

\begin{proof}
The proof of the corollary follows by repeating the proof of Lemma \ref{lem-pseudo:grad-weighted-L^2-estimate} with minor changes. Let us take the function $F$ as in the proof of Lemma \ref{lem-pseudo:grad-weighted-L^2-estimate}. By \eqref{eq:grad-weighted-L^2-estimate-step1} we can write 
\begin{align*}
X_x K_{m(x, G_{\varkappa})} (x,y) & = \frac{1}{2\pi}  \int_{\mathbb{R}} \left\{ \widehat{X_x F}(x, \tilde{\eta}) \right\} \left\{ p_{(1 - i \tilde{\eta}) R^{-2}}(x,y) \right\} d \tilde{\eta} \\ 
& \quad + \frac{1}{2\pi}  \int_{\mathbb{R}} \left\{ \widehat{F}(x, \tilde{\eta}) \right\} \left\{ X_x p_{(1 - i \tilde{\eta}) R^{-2}}(x,y) \right\} d \tilde{\eta}.
\end{align*}

Now, for any fixed $x$, while integrating in $y$-variable over the set $\{ y : d(x,y)<\mathcal{K}_0 \}$, we reduce the power of $d(x,y)$ in the term involving $\widehat{X_x F}(x, \tilde{\eta})$ as follows: 
\begin{align}\label{est:grad-weighted-L^2-estimate-x-grad-estimate}
& \int_{d(x,y)<\mathcal{K}_0} d(x,y)^{2\mathfrak{r}} \left| X_{x} K_{m (x, G_{\varkappa})}(x, y) \right|^2  dy\\
\nonumber & \lesssim_{\mathcal{K}_0, \delta} \sup_{x_0} \int_{d(x,y)<\mathcal{K}_0} d(x,y)^{2 (\mathfrak{r}-(1-\delta))} \left| \int_{\mathbb{R}} \left\{ \widehat{X_x F}(x_0, \tilde{\eta}) \right\} \left\{ p_{(1 - i \tilde{\eta}) R^{-2}}(x,y) \right\} d \tilde{\eta}\right|^2 \, dy \\
\nonumber & \quad + \sup_{x_0} \int_{d(x,y)<\mathcal{K}_0}d(x,y)^{2\mathfrak{r}} \left| \int_{\mathbb{R}} \left\{ \widehat{F}(x_0, \tilde{\eta}) \right\} \left\{ X_x p_{(1 - i \tilde{\eta}) R^{-2}}(x,y) \right\} d \tilde{\eta}\right|^2 \, dy\\
\nonumber & \leq \sup_{x_0} \int_{d(x,y)<\mathcal{K}_0} \frac{(1+ Rd(x,y))^{2 (\mathfrak{r}-(1-\delta))}}{R^{2 (\mathfrak{r}-(1-\delta))}} \left| \int_{\mathbb{R}} \left\{ \widehat{X_x F}(x_0, \tilde{\eta}) \right\} \left\{ p_{(1 - i \tilde{\eta}) R^{-2}}(x,y) \right\} d \tilde{\eta}\right|^2  dy\\
\nonumber & \quad + \sup_{x_0} \int_{d(x,y)<\mathcal{K}_0} \frac{(1+ Rd(x,y))^{2\mathfrak{r}}}{R^{2\mathfrak{r}} } \left| \int_{\mathbb{R}} \left\{ \widehat{F}(x_0, \tilde{\eta}) \right\} \left\{ X_x p_{(1 - i \tilde{\eta}) R^{-2}}(x,y) \right\} d \tilde{\eta}\right|^2 \, dy. 
\end{align}
With the above estimate, one can repeat the rest of the proof of Lemma \ref{lem-pseudo:grad-weighted-L^2-estimate} to complete the proof of estimate \eqref{est:grad-weighted-L^2-estimate-x-grad-compact}. 
\end{proof}

In an analogous manner, one can prove the following $L^\infty$-estimate. 
\begin{corollary} \label{cor-pseudo:grad-unweighted-L-infty-estimate-compact}
For every $\mathfrak{r} \geq 1$, $\epsilon > 0$, $0\leq \delta < 1$, and positive real number $\mathcal{K}_0$, we have 
\begin{align} \label{est:grad-weighted-L-infty-estimate-x-grad-compact} 
& \sup_{x \in \mathbb{R}^{n_1+n_2}} \sup_{d(x,y)<\mathcal{K}_0} |B(x,R^{-1})|^{1/2} |B(y,R^{-1})|^{1/2}  d(x,y)^{\mathfrak{r}} \left| X_{x} K_{m (x, G_{\varkappa})}(x, y) \right| \\
\nonumber & \lesssim_{\mathfrak{r}, \epsilon, \delta, \mathcal{K}_0} \sup_{x_{0}} \left( R^{-(\mathfrak{r}-(1-\delta))} \|X_{x} m(x_0, R^2 \cdot) \|_{W^{\infty}_{\mathfrak{r} - (1 - \delta) + \epsilon}} + R^{-(\mathfrak{r}-1)} \| m(x_0, R^2 \cdot) \|_{W^{\infty}_{\mathfrak{r}+ \epsilon}} \right), 
\end{align}
for every bounded Borel function $m : \mathbb{R}^{n_1 + n_2} \times \mathbb{R} \rightarrow \mathbb{C}$ whose support in the last variable is in $[0, R^2]$ for any $R>0$. 
\end{corollary}

%%%%%%%%%%%%%%%%%%%%%%%%%%%%%%%%%
%%%%%%%%%%%%%%%%%%%%%%%%%%%%%%%%%
%%%%%%%%%%%%%%%%%%%%%%%%%%%%%%%%%
%%%%%%%%%%%%%%%%%%%%%%%%%%%%%%%%%
%%%%%%%%%%%%%%%%%%%%%%%%%%%%%%%%%
%%%%%%%%%%%%%%%%%%%%%%%%%%%%%%%%%
%%%%%%%%%%%%%%%%%%%%%%%%%%%%%%%%%
%%%%%%%%%%%%%%%%%%%%%%%%%%%%%%%%%
%%%%%%%%%%%%%%%%%%%%%%%%%%%%%%%%%

\subsection{Boundedness of Grushin pseudo-multipliers} \label{subsec:proofs-direct-pseudo-grushin} 

We are now in a position to prove Theorems \ref{thm:pseudo-grushin-a=0-unweighted}, \ref{thm:pseudo-grushin-a=0-less-derivative} and \ref{thm:pseudo-grushin-a=0-full-derivative}. 

Given a bounded Borel function $m : \mathbb{R}^{n_1+n_2} \times \mathbb{R}_+ \to \mathbb{C}$ for which $T = m(x, G_{\varkappa}) \in \mathcal{B} \left( L^2 (\mathbb{R}^{n_1+n_2}) \right)$, let us decompose $T = \sum_{j = 0}^\infty T_j$ as follows. 

Choose and fix $\psi_0 \in C_c^\infty((-2,2))$ and $\psi_1 \in C_c^\infty((1/2,2))$ such that $0 \leq \psi_0 (\eta), \psi_1 (\eta) \leq 1$, and 
\begin{align*} 
\sum_{j=0}^\infty \psi_j(\eta) = 1 \quad \text{and} \quad \sum_{j=-\infty}^\infty \psi_1(2^j \eta) = 1, 
\end{align*} 
for all $\eta \geq 0$, where $\psi_j (\eta) = \psi_1 \left( 2^{-(j-1)} \eta \right)$ for $j \geq 2$. 

Now, for each $j \geq 0,$ define $T_j$ to be the Grushin pseudo-multiplier operator, with integral kernel $T_j (x, y) = K_{m_j (x, G_{\varkappa})} (x, y)$, where $m_j (x, \eta) = m (x, \eta) \psi_j (\eta)$. 

%%%%%%%%%%%%%%%%%%%%%%%%%%%%%%%
%%%%%%%%%%%%%%%%%%%%%%%%%%%%%%%

\begin{proof}[Proof of Theorem \ref{thm:pseudo-grushin-a=0-unweighted}]
We are given that $\left| \partial_{\eta}^l m(x, \eta) \right| \leq_{l} (1 + \eta)^{-l}$ for all $l \leq Q + 1$, and also that the operator $T = m(x, G_{\varkappa})$ is bounded on $L^2 ( \mathbb{R}^{n_1 + n_2})$. Note that with $R_0 = Q + \frac{1}{2}$, condition \eqref{cond:General-hypo-sup} follows from Corollary  \ref{cor-pseudo:grad-weighted-L-p-estimate} (with $\Gamma = 0$ and $p=\infty$) and condition \eqref{cond:General-hypo-y-grad-sup} follows from Corollary \ref{cor-pseudo:grad-weighted-L-p-estimate} (with $|\Gamma| = 1$ and $p=\infty$). Hence, it follows from Theorem \ref{thm:weak-type-bound-operator} that $T = m(x, G_{\varkappa})$ is of weak type $(1, 1)$. 
\end{proof}

\begin{proof}[Proof of Theorem \ref{thm:pseudo-grushin-a=0-less-derivative}]
For a fixed $0 \leq \delta < 1$, we are given that 
\begin{align*} 
\left| \partial_{\eta}^{l} m(x, \eta) \right| & \leq_{\Gamma, l} (1 + \eta)^{- l}, \quad \text{for  all} \,\, l \leq \floor*{Q/2} + 1, \\ 
\text{and} \quad \left| X_x \partial_{\eta}^{l} m(x, \eta) \right| & \leq_{\Gamma, l} (1 + \eta)^{-l + \frac{\delta}{2}}, \quad \text{for all} \, \, l \leq \floor*{Q/2}, 
\end{align*} 
and that the operator $T = m(x, G_{\varkappa})$ is bounded on $L^2 ( \mathbb{R}^{n_1 + n_2})$. 

Theorem \ref{thm:pseudo-grushin-a=0-less-derivative} would follow from Theorem \ref{thm:main-sparse} if we could show that the kernels $T_j (x, y) = K_{m_j (x, G_{\varkappa})} (x, y)$ satisfy conditions \eqref{cond:General-hypo} and \eqref{cond:General-hypo-grad} for some $R_0 > \floor*{Q/2}$. To this end, note that condition \eqref{cond:General-hypo} with $R_0 = \floor*{Q/2} + \epsilon$, for any $0 < \epsilon < 1$, follows from Lemma \ref{lem-pseudo:grad-weighted-L^2-estimate} (with $\Gamma = 0$). Finally, for any $0 < \tilde{\epsilon} < 1- \delta$, we can deduce condition \eqref{cond:General-hypo-grad} with $R_0 = \floor*{Q/2} + \tilde{\epsilon}$ from Corollary  \ref{cor-pseudo:grad-unweighted-L-2-estimate-compact}. 
\end{proof} 

%%%%%%%%%%%%%%%%%%%%%%%%%%%%%%%
%%%%%%%%%%%%%%%%%%%%%%%%%%%%%%%

\begin{proof}[Proof of Theorem \ref{thm:pseudo-grushin-a=0-full-derivative}]
For a fixed $0 \leq \delta < 1$, we are given that 
\begin{align*} 
\left| \partial_{\eta}^{l} m(x, \eta) \right| & \leq_{\Gamma, l} (1 + \eta)^{- l}, \quad \text{for  all} \,\, l \leq Q + 1, \\ 
\text{and} \quad \left| X_x \partial_{\eta}^{l} m(x, \eta) \right| & \leq_{\Gamma, l} (1 + \eta)^{-l + \frac{\delta}{2}}, \quad \text{for all} \, \, l \leq Q, 
\end{align*} 
and that the operator $T = m(x, G_{\varkappa})$ is bounded on $L^2 ( \mathbb{R}^{n_1 + n_2})$. 

Theorem \ref{thm:pseudo-grushin-a=0-full-derivative} would follow from Theorem \ref{thm:main-sparse-more-derivative} if we could show that the kernels $T_j (x, y) = K_{m_j (x, G_{\varkappa})} (x, y)$ satisfy conditions \eqref{cond:General-hypo-sup},  \eqref{cond:General-hypo-y-grad-sup} for $R_0 = Q + \frac{1}{2}$ and condition \eqref{cond:General-hypo-grad-sup} for any $R_0 > Q$. For the same, note that with $R_0 = Q + \frac{1}{2}$, condition \eqref{cond:General-hypo-sup} follows from Corollary \ref{cor-pseudo:grad-weighted-L-p-estimate} (with $\Gamma = 0$ and $p=\infty$) and condition \eqref{cond:General-hypo-y-grad-sup} from Corollary \ref{cor-pseudo:grad-weighted-L-p-estimate} (second inequality, with $|\Gamma| = 1$ and $p=\infty$). Finally, for any $0 < \tilde{\epsilon} < 1- \delta$, we can deduce condition \eqref{cond:General-hypo-grad} with $R_0 = \floor*{Q/2} + \tilde{\epsilon}$ from Corollary \ref{cor-pseudo:grad-unweighted-L-infty-estimate-compact}.
\end{proof}

%%%%%%%%%%%%%%%%%%%%%%%%%%%%%%%%%%%%%%%%%%%%%
%%%%%%%%%%%%%%%%%%%%%%%%%%%%%%%%%%%%%%%%%%%%%
%%%%%%%%%%%%%%%%%%%%%%%%%%%%%%%%%%%%%%%%%%%%%
%%%%%%%%%%%%%%%%%%%%%%%%%%%%%%%%%%%%%%%%%%%%%
%%%%%%%%%%%%%%%%%%%%%%%%%%%%%%%%%%%%%%%%%%%%%
%%%%%%%%%%%%%%%%%%%%%%%%%%%%%%%%%%%%%%%%%%%%%
%%%%%%%%%%%%%%%%%%%%%%%%%%%%%%%%%%%%%%%%%%%%%
%%%%%%%%%%%%%%%%%%%%%%%%%%%%%%%%%%%%%%%%%%%%%
%%%%%%%%%%%%%%%%%%%%%%%%%%%%%%%%%%%%%%%%%%%%%
%%%%%%%%%%%%%%%%%%%%%%%%%%%%%%%%%%%%%%%%%%%%%
%%%%%%%%%%%%%%%%%%%%%%%%%%%%%%%%%%%%%%%%%%%%%
%%%%%%%%%%%%%%%%%%%%%%%%%%%%%%%%%%%%%%%%%%%%%
%%%%%%%%%%%%%%%%%%%%%%%%%%%%%%%%%%%%%%%%%%%%%
%%%%%%%%%%%%%%%%%%%%%%%%%%%%%%%%%%%%%%%%%%%%%

\section*{Acknowledgements} 
We are extremely thankful to the anonymous referee for a detailed list of very valuable comments and suggestions which has greatly helped in improving the presentation of the article. We are also indebted to Atreyee Bhattacharya for several wonderful discussions on the sub-Riemannian structure of the Grushin metric. SB and RG were supported in parts from their individual INSPIRE Faculty Fellowships from DST, Government of India. RB was supported by the Senior Research Fellowship from CSIR, Government of India. AG was supported in parts by the  INSPIRE Faculty Fellowship of RG and Institute postdoctoral fellowships from IISER Bhopal and Centre for Applicable Mathematics, TIFR.

%%%%%%%%%%%%%%%%%%%%%%%%%%%%%%%%%%%%%%%%%%%%%
%%%%%%%%%%%%%%%%%%%%%%%%%%%%%%%%%%%%%%%%%%%%%
%%%%%%%%%%%%%%%%%%%%%%%%%%%%%%%%%%%%%%%%%%%%%
%%%%%%%%%%%%%%%%%%%%%%%%%%%%%%%%%%%%%%%%%%%%%
%%%%%%%%%%%%%%%%%%%%%%%%%%%%%%%%%%%%%%%%%%%%%
%%%%%%%%%%%%%%%%%%%%%%%%%%%%%%%%%%%%%%%%%%%%%
%%%%%%%%%%%%%%%%%%%%%%%%%%%%%%%%%%%%%%%%%%%%%
%%%%%%%%%%%%%%%%%%%%%%%%%%%%%%%%%%%%%%%%%%%%%
%%%%%%%%%%%%%%%%%%%%%%%%%%%%%%%%%%%%%%%%%%%%%
%%%%%%%%%%%%%%%%%%%%%%%%%%%%%%%%%%%%%%%%%%%%%
%%%%%%%%%%%%%%%%%%%%%%%%%%%%%%%%%%%%%%%%%%%%%
%%%%%%%%%%%%%%%%%%%%%%%%%%%%%%%%%%%%%%%%%%%%%
%%%%%%%%%%%%%%%%%%%%%%%%%%%%%%%%%%%%%%%%%%%%%
%%%%%%%%%%%%%%%%%%%%%%%%%%%%%%%%%%%%%%%%%%%%%

\providecommand{\bysame}{\leavevmode\hbox to3em{\hrulefill}\thinspace}
\providecommand{\MR}{\relax\ifhmode\unskip\space\fi MR }
\providecommand{\MRhref}[2]{%
  \href{http://www.ams.org/mathscinet-getitem?mr=#1}{#2}
}
\providecommand{\href}[2]{#2}

\end{document}